\documentclass[12pt]{article}
\usepackage[utf8]{inputenc}
\usepackage{amsmath}
\usepackage{amsfonts}
\usepackage{amssymb}
\usepackage{amsthm}
\usepackage{relsize}
\usepackage{array}
\usepackage{multirow}
\usepackage{tabu}
\usepackage{soul}
\usepackage{graphicx}
\usepackage{epigraph}
\usepackage{float}
\usepackage[english]{babel}
\usepackage[usenames, dvipsnames]{color}
\usepackage{fancyhdr}
\usepackage[Sonny]{fncychap}
\usepackage{tikz-cd}
\usepackage{tkz-euclide}
\usepackage{sseq}
\usepackage{newfloat}
\usepackage{amsmath,mathtools}
\usepackage{pgfplots}
\usepackage{geometry}
\usepackage[all]{xy}
\DeclareFloatingEnvironment[fileext=dia]{diagram}
\newtheorem{theorem}{Theorem}
\newtheorem{lemma}{Lemma}
\newtheorem{proposition}[theorem]{Proposition}
\newtheorem{corollary}{Corollary}[theorem]
\theoremstyle{definition}
\newtheorem{definition}{Definition}
\newtheorem{remark}{Remark}
\newtheorem{example}{Example}

\usepackage[most]{tcolorbox}

\usepackage{tabularx}
\usepackage{imakeidx}
\usepackage{hyperref}

\newcommand{\Z}{\mathbb{Z}}
\newcommand{\R}{\mathbb{R}}
\newcommand{\C}{\mathbb{C}}
\newcommand{\Q}{\mathbb{Q}} 
\newcommand{\lie}[1]{\mathfrak{#1}}

\begin{document}

\title{Cohomology of Quotients in Real Symplectic Geometry}

\author{Thomas John Baird, Nasser Heydari}

\maketitle

\begin{abstract}
Given a Hamiltonian system $ (M,\omega, G,\mu) $ where $(M,\omega)$ is a symplectic manifold, $G$ is a compact  connected Lie group acting on $(M,\omega)$ with moment map $ \mu:M \rightarrow\mathfrak{g}^{*}$, then one may construct the symplectic quotient $(M/\!\!/G, \omega_{red})$ where $M/\!\!/G := \mu^{-1}(0)/G$. Kirwan used the norm-square of the moment map, $|\mu|^2$, as a G-equivariant Morse function on $M$ to derive formulas for the rational Betti numbers of $M/\!\!/G$.

A real Hamiltonian system  $(M,\omega, G,\mu, \sigma, \phi) $  is a Hamiltonian system  along with a pair of involutions  $(\sigma:M \rightarrow M, \phi:G \rightarrow G) $  satisfying certain compatibility conditions. These imply that the fixed point set $M^{\sigma}$ is a Lagrangian submanifold of $(M,\omega)$ and that $M^{\sigma}/\!\!/G^{\phi} := (\mu^{-1}(0) \cap M^{\sigma})/G^{\phi}$ is a Lagrangian submanifold of $(M/\!\!/G, \omega_{red})$. In this paper we prove analogues of Kirwan's Theorems that can be used to calculate the $\Z_2$-Betti numbers of $M^{\sigma}/\!\!/G^{\phi} $. In particular, we prove (under appropriate hypotheses) that $|\mu|^2$ restricts to a $G^{\phi}$-equivariantly perfect Morse-Kirwan function on $M^{\sigma}$ over $\Z_2$ coefficients, describe its critical set using explicit real Hamiltonian subsystems, prove equivariant formality for $G^{\phi}$ acting on $M^{\sigma}$, and combine these results to produce formulas for the $\Z_2$-Betti numbers of $M^{\sigma}/\!\!/G^{\phi}$. 
\end{abstract}




\section{Introduction}

\subsection{Motivation and Goals}

A   \textit{Hamiltonian system}  $\mathcal{H}=(M,\omega,G,\mu) $ consists of a Lie group $G$ with Lie algebra $ \mathfrak{g}$ acting on a  symplectic manifold  $(M,\omega)$ via a moment map  $ \mu:M \rightarrow \mathfrak{g}^{*} $. We will always assume that $M$ is connected, that $G$ is both compact and connected and identify $\lie{g} \cong \lie{g}^*$ using an invariant inner product. If $G$ acts freely on the zero level set $ M_{0}:=\mu^{-1}(0) $ then we may construct the \emph{symplectic quotient} $(M/\!\!/G,\omega_{\mathrm{red}})$, where $M/\!\!/G:= M_{0}/G$ is a smooth manifold with symplectic form $\omega_{red}$. 
   
 A \textit{real Hamiltonian system} $ \mathcal{RH}=(M,\omega,G,\mu,\sigma,\phi) $ is a Hamiltonian system $ (M,\omega,G,\mu) $ equipped with an anti-symplectic involution $ \sigma:M \rightarrow M$   ($ \sigma^{2}=\mathrm{Id}_M $ and $ \sigma^{*}\omega=-\omega $) and a Lie group automorphism $ \phi:G \rightarrow G $ of order two ($\phi^2 = \mathrm{Id}_G$)   satisfying certain compatibility conditions (see Definition \ref{def3.try}). These imply that the \emph{real locus} $  M^{\sigma}:=\{ x\in M\ |\ \sigma(x)=x\} $ is a Lagrangian submanifold of $M$ and that the \textit{real subgroup} $  G^{\phi}:=\{ g\in G\ |\ \phi(g)=g \} $ restricts to an action  on $  M^{\sigma}$. If $G$ acts freely on $ M_{0} $, then the \textit{real quotient} $  M^{\sigma}/\!\!/G^{\phi}:=M^{\sigma}_{0}/G^{\phi}$ embeds as a Lagrangian submanifold in the symplectic quotient $M/\!\!/G$. The goal of this paper is to develop Morse theory techniques to calculate the $\Z_2$-Betti numbers of the real quotient $M^{\sigma}/\!\!/G^{\phi} $ in analogy with Kirwan's techniques from \cite{KIR}.
 
In the special case when $G$ is a torus and $\phi(g) = g^{-1}$, a real analogue of Kirwan's equivariant perfection was proven by Goldin-Holm \cite{GH}, and a real analogue of equivariant formality was proven by Biss-Guillemin-Holm \cite{BGH}, building on work of Duistermaat \cite{DUI}. Our paper extends these results to non-abelian real Hamiltonian systems.

Equivariant perfection for non-abelian real Hamiltonian systems has been used in topological gauge theory by Liu-Schaffhauser \cite{LS} and by the first author \cite{BAI} to study moduli spaces of real vector bundles over a real curve. Those papers use a real structure on the Atiyah-Bott \cite{AB2} Yang-Mills Hamiltonian system.

 \subsection{Kirwan's Theorems}
 
 Given an  invariant inner product on the Lie algebra $ \mathfrak{g} $, one can form the norm-square of the moment map
   \begin{equation*}
    f=|\mu|^{2}:M \rightarrow \mathbb{R}.
    \end{equation*}  
Suppose that $f$ is proper. Kirwan \cite{KIR} showed that $f$ is minimally degenerate (or Morse-Kirwan) and used this to derive formulas for the equivariant cohomology. In particular, the critical set $ C_{f} $ of $f$ is a disjoint union of $G$-invariant closed subsets $ \{C_{\beta}\ |\ \beta\in \Lambda\}$ and the function $f$ is $G$-equivariantly perfect over the field of the rational numbers $ \mathbb{Q} $; i.e.
\begin{equation*}
  \mathbf{P}_t^{G}\boldsymbol{(}M;\mathbb{Q}\boldsymbol{)}=\sum_{\beta \in \Lambda} t^{2d_\beta} \mathbf{P}_t^{G}\boldsymbol{(}C_{\beta};\mathbb{Q}\boldsymbol{)},
 \end{equation*} 
where we define the equivariant Poincar\'{e} series relative to a field $ \mathbb{F} $ to be the generating function   $$\mathbf{P}_t^{G}\boldsymbol{(}M,t;\mathbb{F}\boldsymbol{)}:= \sum_{i=0}^{\infty} \dim(H^i_G(M;\mathbb{F})) t^i$$   and $ 2d_\beta $ is the Morse index of $f$ along $ C_{\beta}$ which is necessarily even. If we assume for simplicity that the Morse index is constant along $C_{\beta}$, then the index set $ \Lambda \subset \lie{g}$ is a discrete subset of the Lie algebra and $C_0 = M_0$ is the global minimum.  
  
If additionally $M$ is compact, then Kirwan also proved that  $M$ is $G$-equivariantly formal over the field of the rational numbers $ \mathbb{Q} $; i.e, the Serre spectral sequence of the  fibration $ M \hookrightarrow M_{G}\rightarrow BG $, induced by the homotopy quotient space $ M_{G} $, collapses at page two and thus
\begin{equation*}
 H^{*}_{G}(M;\mathbb{Q})\cong H^{*}(BG;\mathbb{Q}) \otimes_{\mathbb{Q}} H^{*}(M;\mathbb{Q}),
 \end{equation*} 
where $ BG $ is the classifying space of the   Lie group  $G$. Combining these results yields a formula for the rational equivariant Betti numbers of the zero level set $ M_0 := \mu^{-1}(0)$: 

\begin{equation}\label{equ1.orb}
\mathbf{P}_t^G\boldsymbol{(}M_0;\mathbb{Q}\boldsymbol{)}=\mathbf{P}_t\boldsymbol{(}M;\mathbb{Q}\boldsymbol{)} \mathbf{P}_t\boldsymbol{(}BG;\mathbb{Q}\boldsymbol{)}-\sum_{\beta \neq 0} t^{2d_\beta}\mathbf{P}_t^G\boldsymbol{(}C_{\beta};\mathbb{Q}\boldsymbol{)},
\end{equation}
where   $ C_{0}=M_0 $. If $G$ acts freely on the zero level set $ M_{0} $, then $ H_{G}^{*}(M_{0};\mathbb{Q}) = H^{*}(M/\!\!/G;\mathbb{Q}) $, so (\ref{equ1.orb}) can be used to compute the Betti numbers of $M/\!\!/G$.

Kirwan showed that for each $ \beta\in \Lambda $, there is a Hamiltonian subsystem  $  \mathcal{H}_{\beta}=(Z_{\beta},\omega,G_{\beta},\mu_{\beta}) $ in which $ Z_{\beta}\subset M $ is a   symplectic  submanifold, $ G_{\beta} \leq G $ is the stabilizer  subgroup of $\beta $  and  $ C_{\beta}=G \times_{G_{\beta}}   M_{\beta}$ where $ M_{\beta}=\mu_{\beta}^{-1}(0)  $.  This implies $ \mathbf{P}_t^{G_{\beta}}\boldsymbol{(}M_\beta;\mathbb{Q}\boldsymbol{)} =  \mathbf{P}_t^G\boldsymbol{(}C_{\beta};\mathbb{Q}\boldsymbol{)}$ so we obtain the formula

\begin{equation*}
\mathbf{P}_t^G\boldsymbol{(}M_0;\mathbb{Q}\boldsymbol{)}=\mathbf{P}_t\boldsymbol{(}M;\mathbb{Q}\boldsymbol{)} \mathbf{P}_t\boldsymbol{(}BG;\mathbb{Q}\boldsymbol{)}-\sum_{\beta \neq 0} t^{2d_\beta}\mathbf{P}_t^{G_{\beta}}\boldsymbol{(}M_\beta;\mathbb{Q}\boldsymbol{)}.
\end{equation*}
which is recursive in the dimension of $M_\beta$.

\subsection{Summary of Results}
Consider a real Hamiltonian system $ \mathcal{RH}=(M,\omega,G,\mu,\sigma,\phi) $ where $G$ is compact and connected and $f = |\mu|^2$ is proper. Consider the restricted function $ f^{\sigma} := f|_{M^{\sigma}}:M^{\sigma} \rightarrow \mathbb{R} $. We construct a Morse stratification $$ M^{\sigma} = \bigcup_{I \in \mathcal{I}} S^{\sigma}_{I}$$  such that 

\begin{itemize}
\item[(i)]  Each stratum $S^{\sigma}_{I} \subseteq M^{\sigma}$ is a $ G^{\phi}$-invariant locally closed submanifold of constant codimension $d_I$  that equivariantly deformation retracts onto its critical subset $ C^{\sigma}_{I} = C_{f^{\sigma}} \cap S^{\sigma}_I $.

\item[(ii)] There are real Hamiltonian subsystems $\mathcal{RH}_{I}=(Z_{I},\omega,G_{I},\mu_{I},\sigma,\phi)$ such that 
\begin{equation*}
C_{I}^{\sigma}\cong G^{\phi}\times_{G^{\phi}_{I}}  M_I
\end{equation*}
where $M_I := \mu^{-1}_I(0) \subseteq Z_I$.

\end{itemize}

In \S \ref{perfectsect}, we consider equivariant Thom-Gysin  long exact sequence 
\begin{equation*}
\cdots \rightarrow H_{G^{\phi}}^{*-d_{I}}(S^{\sigma}_{I};\mathbb{Z}_{2}) \xrightarrow{i_{I}}  H_{G^{\phi}}^{*}(\bigcup _{J \leq I} S^{\sigma}_{J};\mathbb{Z}_{2}) \rightarrow H_{G^{\phi}}^{*}(\bigcup _{J < I} S^{\sigma}_{J};\mathbb{Z}_{2}) \rightarrow \cdots
\end{equation*}
using an appropriate total order on $\mathcal{I}$. We show that if 
\begin{itemize}
\item[(i)]  $(G, \phi)$ has the free extension property (see \S \ref{freeextsect}), and
\item [(ii)] $\mathcal{RH}$ is $2$-primitive (see \S \ref{perfectsect})
\end{itemize}
then the top equivariant Stiefel-Whitney class of the $G^{\phi}$-equivariant normal bundle of each stratum $ S^{\sigma}_{I} $ is not a zero divisor. This forces $ i_{I} $ to be injective and breaks the long exact sequence into short exact sequences:
 \begin{equation}\label{equ1.vip}
0 \rightarrow H_{G^{\phi}}^{*-d_{I}}(S^{\sigma}_{I};\mathbb{Z}_{2}) \xrightarrow{i_{I}}  H_{G^{\phi}}^{*}(\bigcup _{J \leq I} S^{\sigma}_{J};\mathbb{Z}_{2}) \rightarrow H_{G^{\phi}}^{*}(\bigcup _{J < I} S^{\sigma}_{J};\mathbb{Z}_{2}) \rightarrow 0.
\end{equation}
By induction, this implies that $ f^{\sigma} $ is equivariantly perfect over  the field $ \mathbb{Z}_{2} $; i.e.,
\begin{equation*}
  \mathbf{P}_t^{G^{\phi}}\boldsymbol{(}M^{\sigma}\boldsymbol{)} =\sum_{I} t^{d_{I}} \mathbf{P}_t^{G^{\phi}}\boldsymbol{(}S^{\sigma}_{I}\boldsymbol{)} =  \sum_{I} t^{d_{I}}  \mathbf{P}_t^{G^{\phi}_I}\boldsymbol{(}M^{\sigma}_{I}\boldsymbol{)} 
 \end{equation*} 
 Rearranging yields a recursive formula
 \begin{equation*} \mathbf{P}_t^{G^{\phi}}\boldsymbol{(}M_{0}^{\sigma}\boldsymbol{)}=\mathbf{P}_t^{G^{\phi}}\boldsymbol{(}M^{\sigma}\boldsymbol{)}
   -\sum_{I \neq 0} t^{d_{I}} \mathbf{P}_t^{G^{\phi}_{I}}\boldsymbol{(}M_I^\sigma\boldsymbol{)}.
 \end{equation*} 
and $ \mathbf{P}_t\boldsymbol{(}M^{\sigma}/\!\!/G^{\phi}\boldsymbol{)} = \mathbf{P}_t^{G^{\phi}}\boldsymbol{(}M_{0}^{\sigma}\boldsymbol{)}$ if the action of $G$ on $ M_{0} $ is free. 

In \S \ref{formalitysect} we prove that if $M^{\sigma}$ is compact and the pair $(G,\phi)$ satisfies a \textit{toral free extension property} (Definition \ref{def7.gay}), then $M^{\sigma}$ is an equivariantly formal $G^{\phi}$-space, meaning 
\begin{equation*}
 H^{*}_{G^{\phi}}(M^{\sigma};\mathbb{Z}_{2})\cong H^{*}(BG^{\phi};\mathbb{Z}_{2}) \otimes_{\mathbb{Z}_{2}} H^{*}(M^{\sigma};\mathbb{Z}_{2})
 \end{equation*} 
 and $\mathbf{P}_t^{G^\phi}\boldsymbol{(}M^\sigma\boldsymbol{)} = \mathbf{P}_t\boldsymbol{(}M^\sigma\boldsymbol{)}\mathbf{P}_t\boldsymbol{(}BG^\phi\boldsymbol{)}$.

In \S \ref{examplessect} we apply our results in a number of examples.

\textbf{Acknowledgements:} Thanks to Rebecca Goldin, Hari Kunduri, and Eduardo Martinez-Pedroza for helpful comments on an earlier draft. Thanks also to Jonathan Fischer for advice about Example \ref{ConstructTheGrass}. This research was supported by an NSERC Discovery Grant.

\section{Real Symplectic Geometry}

In this section we recall relevant notions in real symplectic geometry. We suggest Sjamaar \cite{SJA} for a nice introduction to this topic.

\subsection{Hamiltonian actions and symplectic quotients}

\begin{definition}\label{def2.jig}
Let $G$ be a connected Lie group with Lie algebra $ \mathfrak{g} $. Suppose that $G$  acts on a symplectic manifold $ (M,\omega) $. A map $ \mu:M \rightarrow \mathfrak{g}^{*} $ is called a \textbf{moment map} if the following are satisfied.
\begin{enumerate}
\item[(i)]  For any $ X\in \mathfrak{g} $, we have 
\begin{equation*}\label{equ2.lud}
 d \mu^{X}=\iota_{X^{\#}} \omega,
 \end{equation*}
 where $ \mu^{X}:M\rightarrow  \mathbb{R} $ is defined $ \mu^{X}(p)=\langle \mu(p), X\rangle $ and $ X^{\#} $  is the vector field on $M$ generated by the action of $G$; i.e.,
 \begin{equation*}\label{equ2.weu}
  X^{\#}(p)=\frac{d}{dt}\Big |_{t=0} \Big (\exp(tX).p \Big ). 
  \end{equation*}
  \item[(ii)] $ \mu $ is an equivariant map; i.e.,
  \begin{equation*}\label{equ2.woo}
   \mu(g.p)=\mathrm{Ad}^{*}_{g}\Big (\mu(p)\Big ),\ \forall g\in G, \forall p\in M.
   \end{equation*} 
\end{enumerate}
In this case, the tuple $ (M,\omega,G,\mu) $ is called a \textbf{Hamiltonian system}. 
\end{definition}

Let $ (M,\omega,G,\mu) $ be a Hamiltonian system where $G$ is a compact Lie group. Suppose that $ M_{0}=\mu^{-1}(0) $ is the zero level set of $\mu$ and $ i:M_{0} \hookrightarrow M $ is the inclusion map. Since $M_{0}$ is $G$-invariant,   $G$ acts on $M_{0}$. Denote the orbit space $M/\!\!G := M_{0}/G$ and let $q: M_0 \rightarrow M/\!\!G$ the quotient map.   

\begin{proposition}\label{pro2.meu}
If $M_{0}$ is nonempty and $G$ acts freely on $M_{0}$, then the orbit space $ M/\!\!/G $ is  a smooth  manifold with $ \dim M/\!\!/G=\dim M - 2 \dim G $ equipped with a symplectic structure $ \omega_{\mathrm{red}} $ defined by $ i^{*} \omega=q^{*} \omega_{\mathrm{red}} $. The pair $ (M/\!\!/G, \omega_{\mathrm{red}}) $ is called the \textbf{symplectic reduction} or \textbf{symplectic quotient} of $ (M,\omega, G,\mu)$.
\end{proposition}

\begin{proof}
 See \cite{DAS}, Chapter 23.
\end{proof}

\subsection{Real Structures on Symplectic Manifolds}

\begin{definition}\label{def3.orb}
Let $ (M,\omega) $ be a symplectic manifold. An \textbf{anti-symplectic involution}\index{anti-symplectic involution} on $M$ is a diffeomorphism $ \sigma:M \rightarrow M $ having the following properties.
\begin{enumerate}
\item[(1)] $ \sigma^{2}=\mathrm{Id}$.
\item[(2)]  $ \sigma^{*}\omega=-\omega $.
\end{enumerate}
We call $ (M,\omega,\sigma) $   a \textbf{real symplectic manifold}.  The fixed point set  $M^{\sigma}$ is a Lagrangian submanifold of $M$, called a \textbf{real Lagrangian} (see \cite{DUI}).
\end{definition}

\begin{example}\label{exa3.orb}
Equip $\mathbb{C}^n$ with the standard symplectic structure  $ \omega = \frac{\sqrt{-1}}{2} \sum_{i=1}^n dz_i \wedge d\bar{z}_i$. The involution $ \sigma: \mathbb{C}^{n} \rightarrow \mathbb{C}^{n} $ by $ \sigma(z_{1},\dots,z_{n})=(\overline{z}_{1},\dots,\overline{z}_{n}) $ which conjugates each coordinate is an anti-symplectic involution. The real locus is $\mathbb{R}^n$. \end{example}

\begin{example}\label{exatom}
The complex Grassmannian $\mathrm{Gr}_k(\C^n)$ equipped with the standard symplectic form (see Example \ref{ConstructTheGrass}) admits a real structure induced by the complex conjugation involution on $\C^n$ described in Example \ref{exa3.orb}.  The real locus is the real Grassmanian $\mathrm{Gr}_k(\R^n)$.
\end{example}

\begin{example}
Consider a smooth manifold $M$ and denote $ (x,\xi) $ for the coordinates on its cotangent bundle  $ T^{*}M $. Let $ \omega=\sum_{i} dx_{i}\wedge d\xi_{i} $ be the canonical symplectic form on $ T^{*}M $. Suppose there exists a smooth automorphism $ \sigma: M \rightarrow M $ of order two. This extends to an anti-symplectic involution $  \sigma: T^{*}M \rightarrow T^{*}M $ as follows:
\[ \sigma (x,\xi)=(\sigma(x),-\xi\circ d\sigma_{x}).\]
\end{example}

\begin{definition}\label{def3.try}
 A \textbf{real structure}\index{real structure} on a Hamiltonian system $ (M, \omega,G,\mu) $ is a pair of smooth maps $ \sigma: M\rightarrow M $ and $ \phi:G \rightarrow G $ such that the following are satisfied.
\begin{enumerate}
\item[(i)] $ \phi $ is a group involution; i.e., a group automorphism of order two.
\item[(ii)]   $ \sigma $ is an anti-symplectic involution.
\item[(iii)]   $ \sigma $ and $ \phi $ satisfy the following  \textbf{compatibility conditions}\index{compatible conditions}:
\begin{equation}\label{equ3.ali}
\sigma\circ g=\phi(g) \circ \sigma, \ \forall g\in G,\\
 \mu \circ \sigma=-\phi_{*} \circ  \mu.
 \end{equation} 
 Here, $ \phi_{*}=d\phi: \mathfrak{g} \rightarrow \mathfrak{g} $.
\end{enumerate}
We call $ (\sigma,\phi) $ a \textbf{real pair} for $\mathcal{I}$ and the tuple $ \mathcal{RH}=(M,\omega,G,\mu,\sigma,\phi) $ a \textbf{real Hamiltonian system}.
\end{definition}
It follows from Definition \ref{def3.try} the real subgroup  $ G^{\phi}$  restricts to an action on the real locus $ M^{\sigma}$ and on $M^{\sigma}_{0}:=\mu^{-1}(0) \cap M^{\sigma}$. 

\begin{definition}\label{def3.raw}
Let $\mathcal{RH}=(M,\omega,G,\mu,\sigma,\phi) $ be a real Hamiltonian system. The orbit space $ M^{\sigma}/\!\!/G^{\phi}:= M^{\sigma}_{0}/G^{\phi}$  is called the \textbf{real quotient}. 
\end{definition}

The following is due to Foth \cite{F}.

\begin{proposition}
Let $ \mathcal{RH}=(M,\omega,G,\mu,\sigma,\phi) $ be a real Hamiltonian system for which $G$ compact and connected and suppose $G$ acts freely on $M_0 := \mu^{-1}(0)$. Then
\begin{itemize}
\item The anti-symplectic involution  $ \sigma: M \rightarrow M $ descends to an anti-symplectic involution  $\sigma_{\mathrm{red}}:M/\!\!/G \rightarrow M/\!\!/G   $.
\item The real quotient $M^{\sigma}/\!\!/G^{\phi}$ embeds naturally in $M/\!\!/G$ as a union of path components of the real locus $(M/\!\!/G)^{\sigma_{\mathrm{red}}}$.
\end{itemize}
\end{proposition}

The involution $ \phi: G \rightarrow G $ induces an involution $ \phi_{*}: \mathfrak{g} \rightarrow \mathfrak{g} $  determining an eigenspace decomposition 
\begin{equation}\label{equ3.fub}
\mathfrak{g}=\mathfrak{g}_{+}\oplus \mathfrak{g}_{-}.\\
\end{equation}
Clearly $ \mathfrak{g}_{+}=\mathrm{Lie}(G^{\phi}) $ and the adjoint action by $G^{\phi}$ on $\mathfrak{g}$ preserves the decomposion. A consequence of (\ref{equ3.ali}) is that $\mu$ sends $M^{\sigma}$ to $\mathfrak{g}_{-}$.

\begin{example}
The standard action of $\mathrm{U}(n+1)$ on $\mathbb{CP}^n$ is Hamiltonian with respect to the Fubini-Study form $\omega_{FS}$  and moment map $\mu_{FS}: \mathbb{CP}^n \rightarrow \lie{u}(n+1)$ by $$\mu_{FS}[z]=\frac{zz^{*}}{2\pi i |z|^{2}}   .$$ If $M \subseteq \mathbb{CP}^n$ is a non-singular projective variety which is invariant under the action of a connected subgroup $G \leq \mathrm{U}(n+1)$, then this action of $G$ on $M$ is Hamiltonian with respect to the restricted Fubini-Study form on $M$ and moment map $\mu$ obtained by composing $\mu_{FS}$ with the projection $\lie{u}(n+1) \rightarrow \lie{g}$. This action extends to an action by the complexification $G_{\C}$ on $M$ and the Kempf-Ness Theorem states that the Hamiltonian quotient $M/\!\!G$ is isomorphic to  a symplectic manifold with the GIT quotient $M/\!\!G_{\C}$.

Suppose now that $M$ and $G$ are invariant under the standard complex conjugations $\sigma$ of $\mathbb{CP}^n$ and $\phi$ of $\mathrm{GL}_n(\C)$. Then $(M,\omega, G, \mu, \phi, \sigma)$ is a real Hamiltonian system and the quotient $M^{\sigma}/G^{\phi}$ corresponds to a totally real Lagrangian submanifold of $M/\!\!G_{\C}$.
\end{example}

For any $ \beta \in \mathfrak{g} $, let $ \mathcal{O}^{-}_{\beta}=\lbrace \mathrm{Ad}_{g}\beta \ | \ g\in G^{\phi}\rbrace $ and $ \mathcal{O}_{\beta}=\lbrace \mathrm{Ad}_{g}\beta \ |\ g\in G\rbrace $ be the orbits of $ \beta$ with respect to the adjoint actions of $G^{\phi}$ and $G$, respectively.

\begin{lemma}\label{addedlemma}
 Let $G$ be a compact connected Lie group and $\phi \in \mathrm{Aut}(G)$ of order two. For every $G$-orbit $  \mathcal{O}_\beta \subseteq \mathfrak{g} $ the intersection $ \mathcal{O}_{\beta} \cap \mathfrak{g}_{-} $ is a union of finitely many $G^{\phi}$-orbits.
\end{lemma}

\begin{proof}
Clearly $\mathcal{O}_{\beta} \cap \lie{g}_{-}$ is a union of $G^{\phi}$-orbits. We only need to show that the number of orbits is finite. 

Observe that $\lie{g}_- $ is the fixed point set of the linear automorphism $-\phi^*$ of $\lie{g}$, which sends $G$-orbits to $G$-orbits. If $\mathcal{O}_{\beta} \cap \mathfrak{g}_{-}$ is non-empty, then $-\phi^*$ sends $\mathcal{O}_{\beta}$ to itself and $\mathcal{O}_{\beta} \cap \lie{g}_{-} = \mathcal{O}_{\beta}^{-\phi^*}$ is the fixed point set of an order two automorphism. Therefore $\mathcal{O}_{\beta} \cap \lie{g}_{-}$ is union of submanifolds. Given $ \alpha \in  \mathcal{O}_{\beta}\cap \lie{g}_{-} $ we claim that 
\begin{equation}\label{equ4.bra}
T_{\alpha}\mathcal{O}^{-}_{\beta}=(T_{\alpha}\mathcal{O}_{\beta}) \cap \lie{g}_{-}.
\end{equation}
Clearly $T_{\alpha}\mathcal{O}^{-}_{\beta} \subseteq (T_{\alpha}\mathcal{O}_{\beta}) \cap \mathfrak{g}_{-}$. Conversely, suppose that $ \xi \in (T_{\alpha}\mathcal{O}_{\beta}) \cap \lie{g}_{-} $.  Then  $ \xi=[X,\alpha]$, for some $ X\in \mathfrak{g} $. Decompose $ X=X_{+}+X_{-} $ into eigenvectors, so $  [X_{+}, \alpha] \in \mathfrak{g}_{-}^{*}$ and $  [X_{-}, \alpha] \in \mathfrak{g}_{+}^{*}$. Since $\xi \in \lie{g}_-^*$, it follows  that $$ \xi=[X_{+}, \alpha] + [X_{-}, \alpha] = [X_{+}, \alpha] \in T_{\alpha} \mathcal{O}_{\beta} ^{-}.$$

Since both $\mathcal{O}^{-}_{\beta} $ and $\mathcal{O}_{\beta} \cap \lie{g}_{-} $ are manifolds,  (\ref{equ4.bra}) implies that $\mathcal{O}^{-}_{\beta_i} $ is an open subset of $\mathcal{O}_{\beta} \cap \lie{g}_{-} $. Since $\mathcal{O}_{\beta} \cap \lie{g}_{-} $ is compact, it must be covered by a finite number of them, completing the proof. 
\end{proof}

\section{Morse stratification for Real Hamiltonian systems}

\subsection{Morse Stratification for Hamiltonian systems}

Let $\mathcal{H}=(M,\omega, G, \mu) $ be a Hamiltonian system in which   $G$ is  compact and connected and $M$ is connected of dimension $2n$. Fix an $\mathrm{Ad}$-invariant inner product on the Lie algebra $ \mathfrak{g} $ which identifies $\lie{g} \cong \lie{g}^*$ and fix a $G$-invariant Riemannian metric on $M$ compatible with $ \omega$. Define $f:M \rightarrow \mathbb{R}  $ by
\begin{equation}\label{equ4.wsw}
 f(p)=|\mu(p)|^{2}, \ \forall p\in M,
\end{equation}
and assume that $f$ is proper. This ensures that the negative gradient flow determined by the vector field $- \nabla f$ exists for all positive time.

For any $ \beta\in \lie{g}$, the component map $\mu^{\beta}:M \rightarrow \mathbb{R} $ is a (not necessarily proper) Morse-Bott function whose critical set  $ C_{\mu^{\beta}} := \{x \in M \ |\ d\mu^{\beta}_x =0 \}$ is a union of symplectic submanifolds with even Morse indices. Let $ Z_{\beta} $ be the union of path components of $ C_{\mu^{\beta}} $ on which $ \mu^{\beta} $ takes the value $ |\beta|^{2} $; i.e.
\begin{equation}\label{equ4.orb}
Z_{\beta} :=C_{\mu^{\beta}} \cap \Big [ (\mu^{\beta})^{-1}(|\beta|^{2})\Big ].
\end{equation}
Decompose
$$  Z_{\beta} = \coprod_{m=0}^n Z_{\beta,m}$$
where $Z_{\beta,m}$ is the set of points with Morse index $2m$. Denote the isotropy group $ G_{\beta}=\{g\in G\ |\ \mathrm{Ad}_{g}\beta=\beta\} $ with Lie algebra $ \mathfrak{g}_{\beta}$. Then each $ Z_{\beta,m} $ is $ G_{\beta} $-invariant and $$\mathcal{H}_{\beta,m} = (Z_{\beta,m}, \omega, G_{\beta},\mu_{\beta,m})$$ is a Hamiltonian system with moment map $\mu_{\beta,m} = \mathrm{Pr}_{\beta} \circ \mu|_{Z_{\beta,m}}-\beta$ where $ \mathrm{Pr}_{\beta}: \lie{g}\rightarrow \lie{g}_{\beta} $ is orthogonal projection.
The following are due to Kirwan \cite{KIR} (except K2 which is an improvement of Kirwan's result due to Duistermaat-Lerman \cite{LER}).

\begin{enumerate}

\item[K1] The critical set $ C_{f} $  of $f$ is a finite union of disjoint, $G$-invariant, compact (possibly disconnected) subsets 
 \begin{equation*}\label{equ4.zip}
C_{f}=\coprod_{\beta,m} C_{\beta,m},
\end{equation*}
where $\beta \in \lie{t}_+$ are elements of the positive Weyl chamber and $ f(C_{\beta,m}) = \mathcal{O}_\beta$ is the (co)adjoint orbit of $\beta$ and $m \in  \{ 0,1,\dots, \dim(M)/2\} $.
\\
 
 \item[K2] $f$ determines a G-invariant Morse stratification into locally closed submanifolds
$$ M = \bigcup_{\beta,m} S_{\beta,m} $$
where  $S_{\beta,m} $ deformation retracts onto $C_{\beta,m}$ under the negative gradient flow by $f$. Partially order the indices by
\begin{equation*}
(\beta_{1},m_1) < (\beta_{2},m_2) \Leftrightarrow |\beta_{1}| < |\beta_{2}|. 
\end{equation*}
The closure of a stratum $S_{\beta,m}$ in $M$  satisfies $$ \overline{S}_{\beta,m} \subseteq \bigcup_{(\gamma, k) \geq (\beta, m)} S_{\gamma,k} $$
\\
 
\item[K3] Given $p \in C_{\beta,m}$, the tangent space $T_pS_{\beta,m}$ is a symplectic vector subspace of $T_pM$. The codimension of $S_{\beta,m}$ is even and equal to  $$2d(\beta,m) := 2m- \dim G +\dim G_{\beta}.$$ 
\\

\item[K4] For any $ \beta $ and $m$, the Hamiltonian subsystem $ \mathcal{H}_{\beta,m}=(Z_{\beta,m}, \omega, G_{\beta},\mu_{\beta,m}) $ satisifes  $$ \mu^{-1}_{\beta,m}(0)=Z_{\beta,m}\cap \mu^{-1}(\beta) $$ and 
\begin{equation*}\label{equ4.try}
 C_{\beta,m}= G \cdot   \mu_{\beta,m}^{-1}(0) \cong G \times_{G_{\beta}} \mu_{\beta,m}^{-1}(0).
 \end{equation*} 
\\

\end{enumerate}

\subsection{Morse Stratification for real Hamiltonian systems}\label{MSfRH}

Consider a real Hamiltonian system $ \mathcal{RH}=(M,\omega,G,\mu,\sigma,\phi) $ where $G$ is compact and connected and $M$ is connected. Choose invariant compatible metrics on $M$ and $\lie{g}$ so that $f = |\mu|^2$ is $\sigma$-invariant and suppose that $f$ is proper. Let $$f^{\sigma}: M^\sigma \rightarrow \R$$ be the restriction of $f$ to the real locus $M^{\sigma}$. This implies that
\begin{equation}\label{invtgrad}
\nabla f = \sigma_*(\nabla f)
\end{equation} 
and in particular that along $M^{\sigma}$, we have equality $\nabla f = \nabla f^{\sigma}$. Therefore the negative gradient flow on $M$ preserves $M^{\sigma}$.

\begin{proposition}\label{pro4.wsw}
Let $ \mathcal{RH}$ and $f^{\sigma}$ be as above. Then
\begin{enumerate}

\item[(i)]    The critical set of $f^{\sigma}$ is $ C_{f^{\sigma}}= C_{f}\cap M^{\sigma}$ and  thus
\begin{equation*}\label{equ4.nos}
C_{f^{\sigma}}=\coprod_{\beta,m} C_{\beta,m} ^{\sigma}, 
\end{equation*}
where $ C_{\beta,m} ^{\sigma}=C_{\beta,m}  \cap M^{\sigma} $ are $ G^{\phi}$-invariant closed subsets of $M^{\sigma}$ on each of which $ f^{\sigma}$ takes a constant value.
\\

\item[(ii)]  We have a stratification into locally closed submanifolds
$$ M^{\sigma} = \bigcup_{\beta,m} S_{\beta,m}^\sigma$$
where $ S^{\sigma}_{\beta,m} =S_{\beta,m}  \cap M^{\sigma} $ and  $ S_{\beta,m} ^{\sigma}$ deformation retracts onto $C^{\sigma}_{\beta,m}$ via the negative gradient flow. 
\\
\item[(iii)] The closure of $S_{\beta,m}^{\sigma}$ in $M^{\sigma}$ satisfies
$$ \overline{S^{\sigma}}_{\beta,m} \subseteq \bigcup_{(\gamma, k) \geq  (\beta, m)} S_{\gamma,k}^{\sigma} .$$

\item[(iv)]  The codimension of $S_{\beta,m}^{\sigma}$ in $M^{\sigma}$ is half the codimension of $S_{\beta,m}$ in $M$.
\end{enumerate}

\end{proposition}

\begin{proof}
Both (i) and (ii) follows easily from Kirwan's Theorems K1, K2 combined with (\ref{invtgrad}). 

The closure of $S_{\beta,m}^{\sigma}$ satisfies $$\overline{S^{\sigma}}_{\beta,m}  \subseteq \overline{S}_{\beta,m} \cap M^{\sigma} \subseteq \Big(\bigcup_{(\gamma, k) \geq (\beta, m)} S_{\gamma,k} \Big) \cap M^{\sigma} = \bigcup_{(\gamma, k) \geq (\beta, m)} S_{\gamma,k}^{\sigma}$$ proving (iii).

To prove (iv) recall that at $p \in C_{\beta,m}^{\sigma}$ the tangent space $T_pS_{\beta,m}$ is a symplectic vector subspace of $T_pM$. Because $\sigma$ is anti-symplectic it follows that $ T_pS_{\beta,m}^{\sigma}$ is Lagrangian in $T_pS_{\beta, m}$, hence half dimensional. Since $M^{\sigma}$ has half the dimension of $M$ the result follows.
\end{proof}

As we saw in Kirwan's result K4, for each $ \beta $ and $ m $, there exists a  Hamiltonian subsystem $\mathcal{H}_{\beta,m}= (Z_{\beta,m},\omega,G_{\beta},\mu_{\beta,m}) $. It is easy to see that  $ (\sigma,\phi) $ also restricts to a real pair on this system. 

\begin{proposition}\label{pro4.orb}\label{pro4.zip}
Choose $G^{\phi}$-orbits representatives $\beta_1,\dots,\beta_k \in \mathcal{O}_\beta  \cap \lie{g}_-$ (these are finite by Lemma \ref{addedlemma}). Then for each $m$ we have a natural $G^{\phi}$-equivariantly diffeomorphism  
$$C_{\beta,m}^{\sigma} \cong \coprod_{i=1}^k G^{\phi} \times_{G^\phi_{\beta_i}}\mu^{-1}_{\beta_i,m}(0)^{\sigma}.$$
\end{proposition}

\begin{proof}
Since $\mu(C_{\beta,m}^{\sigma} ) \subseteq \mathcal{O}_\beta \cap \lie{g}_- = \coprod_{i=1}^k \mathcal{O}^-_{\beta_i}$ is a disconnected union, we obtain $$C_{\beta,m}^{\sigma} = \coprod_{i=1}^k [C_{\beta,m}^{\sigma} \cap \mu^{-1}(\mathcal{O}^-_{\beta_i})].$$

By K4, we have a $G$-diffeomorphism $$ C_{\beta,m} = C_{\beta_i,m} \cong G \times_{G_{\beta_i}} \mu^{-1}_{\beta_i,m}(0).$$ The subset $C_{\beta_i,m}^{\sigma}$ corresponds to those equivalence classes of pairs $(g,x)\in G \times \mu^{-1}_{\beta_i,m}(0)$ for which there exists $h \in G_{\beta_i}$ such that $(g h^{-1}, h x) = (\phi(g), \sigma(x))$. 

The subset $C_{\beta_i,m}^{\sigma}\cap \mu^{-1}(\mathcal{O}^-_{\beta_i})$ corresponds to those equivalence classes containing a representative $(g,x)$ for which $g = \phi(g)$. Such a class is fixed by $\sigma$ if and only if $(g h^{-1}, h x) = (\phi(g), \sigma(x)) = (g,\sigma(x))$, so $(g,x) \in G^{\phi} \times \mu_{\beta_i,m}^{-1}(0)$. Finally, if $g \in G^{\phi}$ then $gh^{-1} \in G^{\phi}$ if and only if $h \in G^{\phi}$. We conclude that 
$$  C_{\beta,m}^{\sigma} \cap \mu^{-1}(\mathcal{O}^-_{\beta_i}) \cong G^{\phi} \times_{G^\phi_{\beta_i}}( \mu^{-1}_{\beta_i,m}(0)^\sigma).$$

\end{proof}

\section{The Free Extension Property}

\subsection{Elementary abelian 2-subgroups}

\begin{definition}\label{def2.hsh}
Let $G$ be a compact Lie group. An \textbf{elementary abelian $2$-subgroup} is a subgroup $E \leq G$ which is isomorphic to $(\mathbb{Z}_{2})^{n}$ for some $n \geq 0$. We say such an $E$ is \textbf{maximal} if it is not contained in a larger elementary abelian $p$-subgroup.
\end{definition}

 \begin{example}\label{exa2.gal}
The $n$-torus $\mathbb{T}^{n}= \mathrm{U}(1)^{n}$ contains a unique maximal abelian $2$-subgroup, $E(n) =\{\lambda\in T\ |\ \lambda^{2}=1\} = \mathrm{O}(1)^n \cong (\Z_2)^n$. 
 \end{example}

\begin{example}\label{exa2.rsa}
The diagonal matrix group 
\begin{equation*}\label{equ2.jus}
D(n)= \Big\lbrace \begin{pmatrix}
\varepsilon_{1} & \cdots & 0\\
\vdots & \ddots & \vdots \\
0 & \cdots & \varepsilon_{n} \\
\end{pmatrix}\ |\ \varepsilon_{i}\in \lbrace \pm 1\rbrace \Big \rbrace \cong (\Z_2)^n
\end{equation*}
is the unique maximal elementary abelian 2-group up to conjugation in both $\mathrm{O}(n)$ and $\mathrm{U}(n)$. 
\end{example}

\begin{example}\label{exa2.ilo}
The matrix group
\begin{equation*}
SD(n)= \Big\lbrace \begin{pmatrix}
\varepsilon_{1} & \cdots & 0\\
\vdots & \ddots & \vdots \\
0 & \cdots & \varepsilon_{n} \\
\end{pmatrix} \in D(n) \ |\  \prod_{i=1}^n \varepsilon_{i} =1  \Big \rbrace \cong (\Z_2)^{n-1}
\end{equation*}
is the unique maximal elementary abelian 2-group up to conjugation in both $\mathrm{SO}(n)$ and $\mathrm{SU}(n)$. 
\end{example}

In general, a compact connected  Lie group $G$ may contain more than one conjugacy class of maximal elementary abelian 2-groups. We refer to Griess \cite{Gri} for a classification.

\subsection{Free extensions and spectral sequences}

\begin{definition}\label{def2.tic}
Let $X$ be a left $G$-space  and $ EG\rightarrow BG $ be a fixed universal $G$-bundle. The twisted product space $ X_{G}=EG \times_{G} X $ is called the \textbf{homotopy quotient} or the  \textbf{Borel construction}\index{homotopy quotient}\index{Borel construction} of $X$ with respect to the fixed universal bundle. Given a commutative ring with unity $R$ define the \textbf{equivariant cohomology}
\begin{equation*}
H^{*}_{G}(X;R)=H^{*}(X_{G};R).
\end{equation*}
\end{definition} 

Let $i: K \hookrightarrow G$ be an inclusion of compact Lie groups, and suppose $G$ acts on a connected manifold $X$. Then the homotopy quotients fit into the natural commutative diagram:
\begin{equation}\label{pullbackdiag}
\xymatrix{
G/K \ar[d]^{=} \ar[r]^{j_X}  & X_{K} \ar[r]^{i_{X}}  \ar[d] & X_{G} \ar[d] \\
G/K  \ar[r]^{j} & BK \ar[r]^{Bi} & BG}
\end{equation}

\begin{proposition}\label{pro5.orb}
Let $G$, $K$, and $X$ be as above and let $\mathbb{F}$ be a field. If $ j^{*}:H^{*}(BK;\mathbb{F}) \rightarrow H^{*}(G/K) $ is surjective then we have an isomorphism  $$ H_{K}^{*}(X;\mathbb{F}) \cong H^*(G/K;\mathbb{F}) \otimes_K H_{G}^{*}(X;\mathbb{F}) $$ as graded $H_{G}^{*}(X;\mathbb{F})$-modules. In particular, $ i_{X}^{*} : H_{G}^{*}(X;\mathbb{F}) \rightarrow  H_{K}^{*}(X;\mathbb{F})  $ is injective.
\end{proposition}

\begin{proof}
If $j^*$ is surjective, then $j_X^* : H_{K}^{*}(X;\mathbb{F}) \rightarrow H^*(G/K;\mathbb{F})$ must also be surjective by applying cohomology to the commutative diagram (\ref{pullbackdiag}). The result now follows by the Leray-Hirsch Theorem. 
\end{proof}

In the following proposition, we give a criterion for the surjectivity of the morphism $ j^{*} $ in the fibration   $G/K\overset{j}{\hookrightarrow}  BK \overset{Bi}{\rightarrow} BG $.  
Recall that a ring homomorphism $ \varphi: R \rightarrow S $ between commutative rings with unity is called a \textbf{free extension} if $S$ is a free $R$-module with respect to the module structure induced by $\varphi$.
 
\begin{proposition}\label{pro5.aba}
Let $K\leq G $  be a pair of compact Lie groups and $\mathbb{F}$ be a field. Then the following are equivalent
\begin{enumerate}
\item[(i)] $ j^{*}:H^{*}(BK;\mathbb{F}) \rightarrow H^{*}(G/K) $ is surjective.
\item[(ii)]  $Bi^*: H^{*}(BG;\mathbb{F}) \rightarrow H^*(BK;\mathbb{F})  $ is a free extension and the action of $ \pi_1(BG) $ on $H^{*}(G/K;\mathbb{F})$ is trivial.
\end{enumerate}
\end{proposition}
\begin{proof}
That (i) implies (ii) is proven in (\cite{MT} Theorem 4.4). Conversely, suppose $Bi^*: H^{*}(BG;\mathbb{F}) \rightarrow H^*(BK;\mathbb{F})  $ is a free extension and the action of $ \pi_1(BG) $ on  $H^{*}(G/K;\mathbb{F})$ is trivial. The Eilenberg-Moore spectral sequence of this fibration converges strongly to $ H^*(G/K;\mathbb{F}) $ and its second page has the following form:
$$E_2^{*,*}\cong \mathbb{F}\otimes_{H^*(BG;\mathbb{F})} H^*(BK;\mathbb{F}). $$
The free extension condition forces this spectral sequence to collapse to the zeroth column $ E_2^{0,*} $ which implies that 
$$
H^*(G/K;\mathbb{F})\cong \mathbb{F}\otimes_{H^*(BG;\mathbb{F})} H^*(BK;\mathbb{F}).
$$
 We can write
 \begin{align*}
 H^*(BK;\mathbb{F})&\cong  H^*(BG;\mathbb{F})  \otimes_{\mathbb{F}} \mathbb{F} \otimes_{H^*(BG;\mathbb{F})}  H^{*}(BK;\mathbb{F})\\
 & \cong H^*(BG;\mathbb{F}) \otimes_{\mathbb{F}}   H^*(G/K;\mathbb{F}). 
 \end{align*}
  That is, the Serre spectral sequence of the fibration collapses at page two and by the Leray-Hirsch theorem, $ j^*: H^*(BK;\mathbb{F}) \rightarrow  H^*(G/K;\mathbb{F})$ is surjective. 
\end{proof}

\begin{example}\label{exa.rst}
Let $T:=\mathrm{U}(1)^{n}$ be an $n$-torus and $ E(n):= \mathrm{O}(1)^n$ be its unique maximal elementary abelian 2-subgroup. Consider the fibration $ T/E(n)\overset{j}{\hookrightarrow} BE(n) \overset{Bi}{\rightarrow} BT $. The cohomology of classifying spaces $ BT $ and $ BE(n) $ are 
$$
H^{*}(BT;\mathbb{Z}_{2})=\mathbb{Z}_{2}[c_{1},\dots,c_{n}]
$$
and
$$
H^{*}(BE(n);\mathbb{Z}_{2})=\mathbb{Z}_{2}[x_{1},\dots,x_{n}]
$$
 where $c_{i}$ is the Chern class of degree $i$,  $x_{i}$ has degree one and  the  induced morphism $ Bi^{*}: H^{*}(BT;\mathbb{Z}_{2}) \rightarrow  H^{*}(BE(n);\mathbb{Z}_{2})$  sends $ c_{i} $ to $ x^{2}_{i} $ (see \cite{MT}, Theorem 5.11). This shows that the set $$\{x_{1}^{m_{1}}...x_{n}^{m_{n}} \ |\ m_{k}\in \{0,1\} \}$$ is a basis for $H^{*}(BE(n);\mathbb{Z}_{2})  $ as a $H^{*}(BT;\mathbb{Z}_{2})  $-module. Therefore, $ Bi^{*} $ is a free extension. Since $T$ is connected, $ \pi_1(BT) $ is trivial and by Proposition \ref{pro5.aba}  the induced morphism $ j^*: H^*(BE(n);\Z_2) \rightarrow  H^*(T/E(n);\Z_2)$ is surjective.     
\end{example}

\begin{example}\label{exa2.spt}\label{exa5.try}
Let $ D(n)\subseteq \mathrm{O}(n) $ and $ SD(n)\subseteq \mathrm{SO}(n) $ be the subgroups of diagonal matrices respectively. Consider two fibrations $ \mathrm{O}(n)/D(n) \overset{j}{\hookrightarrow} BD(n) \overset{Bi}{\rightarrow} B\mathrm{O}(n) $
and $ \mathrm{SO}(n)/SD(n) \overset{j_0}{\hookrightarrow} BSD(n) \overset{Bi_0}{\rightarrow} B\mathrm{SO}(n) $.   It is known by Theorem 5.9 in \cite{MT}  that the induced morphisms $$j^*: H^{*}(BD(n);\Z_{2}) \rightarrow H^{*}(\mathrm{O}(n)/D(n);\Z_{2}) $$ and $$j_{0}^{*}: H^{*}(BSD(n);\Z_{2}) \rightarrow H^{*}(\mathrm{SO}(n)/D(n);\Z_{2}) $$  are both surjective. 
\end{example}

\subsection{The Free Extension Property for Involutive Lie Groups}\label{freeextsect}

An \textbf{involutive Lie group} is a pair $ (G,\phi) $ in which $G$ is a Lie group  and $ \phi: G \rightarrow G $ is an automorphism of order 2. The subgroup of invariant elements is denoted by $ G^{\phi}=\{g\in G\ |\ \phi(g)=g\} $.
Decompose the Lie algebra into $\phi_*$-eigenspaces $ \mathfrak{g}=\mathfrak{g}_{+}\oplus \mathfrak{g}_{-} $.  Given $ \beta\in \mathfrak{g}_{-} $, consider
\begin{equation*}
 G^{\phi}_{\beta}:=\lbrace g\in G^{\phi}\ |\ \mathrm{Ad}_{g} \beta=\beta \rbrace.
 \end{equation*}

\begin{definition}\label{def5.orb}
We say the involutive Lie group $ (G,\phi) $ has the \textbf{free extension property} (FEP)  if for all $ \beta\in \mathfrak{g}_{-} $ and every maximal abelian 2-subgroup $D_\beta \leq G^{\phi}_{\beta} $, the induced morphism  $ H^{*}(BD_{\beta};\mathbb{Z}_{2}) \rightarrow H^*(G^{\phi}_\beta/D_\beta;\Z_2)$ is surjective. 
\end{definition}

\begin{remark}\label{rem5.aba}
The terminology ``free extension" comes from Proposition \ref{pro5.aba}.
\end{remark} 

\begin{remark}\label{rem5.orb}
An involutive Lie group $ (G,\phi) $ has FEP if and only if  $ (G_{\beta}^{\phi}, \mathrm{Id}) $ has FEP for every $ \beta\in \mathfrak{g}_{-} $.
\end{remark} 
 
\begin{proposition}\label{pro.abc}
Let $G$ be a compact connected Lie group. If the integral cohomology $H^*(G;\Z)$ contains no 2-torsion, then the pair $ (G,\mathrm{Id}) $ has the free extension property.
\end{proposition}
\begin{proof}
Let $D$ be a maximal elementary abelian 2-subgroup of $G$. By a theorem of Borel (see \cite{KN}, Theorem 1.1), since   $ H^{*}(G;\Z) $ has no 2-torsion, there exists a maximal torus $T$ in $G$ such that $ D \leq  T$.  The morphism $j: G/D \rightarrow BD$ fits into a commuting diagram 
$$  \xymatrix{ T/D \ar[r]^k \ar[d]^=& G/D \ar[r] \ar[d]^j & G/T \ar[d]^{j'} \\ 
T/D \ar[r]^i & BD \ar[r] & BT}$$
which is a pullback of $T/D$ fibre bundles. By Example \ref{exa.rst} we know that $i^*$ is surjective, from which it follows that $k^*$ is surjective. Therefore by Leray Hirsch we have $ H^*(G/D;\Z_2) \cong H^*(T/D;\Z_2) \otimes H^*(G/T;\Z_2)$ and  $ H^*(BD;\Z_2) \cong H^*(T/D;\Z_2) \otimes H^*(BT;\Z_2)$ and these isomorphisms identify $j^*$ with $Id_{H^*(T/D;\Z_2)} \otimes (j')^*$. Since $(j')^*$ is known to be surjective by  Theorem 8.3 in \cite{MCC}, we conclude that $j^*$ is surjective.
\end{proof}

\begin{proposition}\label{rem5.wsw}
The following pairs   $$(\mathrm{U}(n),\mathrm{Id}), (\mathrm{SU}(n),\mathrm{Id}), (\mathrm{Sp}(n),\mathrm{Id}), (\mathrm{O}(n),\mathrm{Id}), (\mathrm{SO}(n),\mathrm{Id})$$ have the free extension property.
\end{proposition} 

\begin{proof}
Since the integral cohomology of compact connected Lie groups $ \mathrm{U}(n) $, $ \mathrm{SU}(n) $ and $\mathrm{Sp}(n)$ have no 2-torsion (see \cite{MT}, Corollary 3.11), Proposition \ref{pro.abc} follows that $(\mathrm{U}(n),\mathrm{Id})$, $(\mathrm{SU}(n),\mathrm{Id})$, $(\mathrm{Sp}(n),\mathrm{Id})$ have free extension property. The result for $(\mathrm{O}(n),\mathrm{Id})$ and  $(\mathrm{SO}(n),\mathrm{Id})$ was established in Example \ref{exa5.try}.
\end{proof}

\begin{proposition}\label{pro5.try}
If $ (G,\phi) $ and $ (H,\psi) $ have the free extension property, then $ (G \times H,\phi \times \psi) $ also has the free extension property. 
\end{proposition}
\begin{proof}
Given $ \beta=\beta_{1}+\beta_{2}\in (\mathfrak{g} \oplus \lie{h})_{-} = \mathfrak{g}_{-}  \oplus \mathfrak{h}_{-} $ then
 \begin{equation*} (G \times H)^{\phi \times \psi}_{\beta} = G_{\beta_{1}}^{\phi} \times H_{\beta_{2}}^{\psi}.
 \end{equation*}
 A maximal abelian 2-subgroup in $G_{\beta_{1}}^{\phi} \times H_{\beta_{2}}^{\psi}$ must be a product $D_{\beta_1} \times D_{\beta_2}$ of a maximal abelian 2-subgroups of the factors. Applying the Kunneth Theorem to $ G_{\beta_{1}}^{\phi}/D_{\beta_1}  \times H_{\beta_{2}}^{\psi}/D_{\beta_2} \rightarrow BD_{\beta_1} \times BD_{\beta_2} $ completes the proof.
\end{proof}

\begin{proposition}\label{pro5.zip}
Both $ (\mathrm{U}(n),\phi) $ and $ (\mathrm{SU}(n),\phi) $ have the free extension property, where $ \phi(A)=\overline{A} $ is entry-wise complex conjugation. 
\end{proposition}

\begin{proof}
We start with the case $\mathrm{U}(n)$. We have $\mathrm{U}(n)^{\phi} = \mathrm{O}(n)$ and $\sqrt{-1} \beta$ is a real symmetric matrix for every $\beta \in \mathfrak{u}(n)_{-}$. By the Spectral Theorem, $\beta$ is orthogonally diagonalizable with imaginary eigenvalues. Therefore

\begin{equation*}
\mathrm{O}(n)_{\beta} \cong \mathrm{O}(k_{1}) \times \cdots \times \mathrm{O}(k_{p}).
\end{equation*}
where $k_1,\dots,k_p$ are the multiplicities of the eigenvalues of $\beta$, so $k_1+\dots,+k_p =n$. The result now follows from Propositions \ref{rem5.wsw}, \ref{pro5.try} and Remark \ref{rem5.orb}.

In the case $\mathrm{SU}(n)$ we obtain similarly $\mathrm{SU}(n)^{\phi} = \mathrm{SO}(n)$  and 

\begin{equation*}\label{equ5.ilo}
\mathrm{SO}(n)_{\beta} \cong  \mathrm{SO}(n) \cap (\mathrm{O}(k_{1}) \times \cdots \times \mathrm{O}(k_{p})).
\end{equation*}

Given a maximal elementary abelian $2$-group $SD \leq \mathrm{SO}(n)_{\beta}$, there exists a maximal elementary abelian 2-group $D \subseteq \mathrm{O}(n)_{\beta} $ such that $SD=D\cap \mathrm{SO}(n)_{\beta}$.
We obtain an equality of quotients $$F := \mathrm{O}(n)^{\phi}_{\beta}/D = \mathrm{SO}(n)_{\beta}/SD, $$ and a pullback of fibre bundles

\begin{equation}\label{dia5.zip}
 \xymatrix{\mathrm{SO}(n)_{\beta}/SD \ar[r]^{=}\ar[d]^{j} &\mathrm{O}(n)^{\phi}_{\beta}/D\ar[d]^{j'}  \\
BSD\ar[d]  \ar[r]  &BD \ar[d]\\
B\mathrm{SO}(n)_{\beta}  \ar[r]  &B\mathrm{O}(n)_{\beta} \\} 
 \end{equation} 
so surjectivity of $j^*$ follows from the surjectivity of $j'^*$ by commutativity.
\end{proof}

 \begin{proposition}\label{exa5.ali}
 Let $ G=\mathrm{U}(n)\times \mathrm{U}(n) $ and $ \phi: G \rightarrow G  $ be defined by $ \phi(A,B)=(\overline{B}, \overline{A}) $.  Then $(G,\phi)$ satisfies the free extension property.
 \end{proposition}
 
 \begin{proof}
 We have $G^{\phi}=\{(A, \overline{A})\ |\ A\in \mathrm{U}(n)\}\cong \mathrm{U}(n)$ and $ \mathfrak{g}_{-}=\{ (X, -\overline{X}) \ |\ X\in \lie{u}(n)\}$.

Given $ \beta=  (X, -\overline{X})\in \mathfrak{g}_{-} $, observe that $AX A^{-1} = X$ if and only if $\overline{A} (-\overline{X}) \overline{A}^{-1} = -\overline{X}$.  Therefore $G_\beta^{\phi}$ is isomorphic to the centralizer of $X$ in $\mathrm{U}(n)$. Since $X$ is skew-Hermitian, the Spectral Theorem implies $X$ is orthogonally diagonalizable, so  
 \begin{equation}\label{equ5.gel}
  G^{\phi}_{\beta} \cong \mathrm{U}(k_{1}) \times \cdots \times \mathrm{U}(k_{p}).
  \end{equation} 
for some natural numbers $ k_{1},\dots,k_{p} $ with $ k_{1}+\cdots +k_{p}=n $. The proof now follows from Propositions \ref{rem5.wsw}, \ref{pro5.try} and Remark \ref{rem5.orb}.
\end{proof}

 \begin{proposition}\label{exa5.ali}
Consider the projective unitary group $ G=\mathrm{PU}(4n+2)$ for $n \geq 1$ and $ \phi: G \rightarrow G  $ be defined by $ \phi(A)= \overline{A} $.  Then $(G,\phi)$ does \underline{not} satisfy the free extension property.
 \end{proposition}
 
\begin{proof}
A necessary condition for $(G,\phi)$ to have the free extension property is that $H^*(BG^\phi_\beta;\Z_2)$ be an integral domain for all $\beta \in \lie{g}_-$, for otherwise it can't possibly inject into the polynomial ring $H^*(BD_{\beta}; \Z_2)$. 
The real subgroup is $\mathrm{PU}(4n+2)^{\phi} = \mathrm{PO}(4n+2)$ and the ring $H^*(B\mathrm{PO}(4n+2);\Z_2)$ contains zero divisors (see \cite{KM} Proposition 5.8).
 \end{proof}

\section{An Atiyah-Bott Lemma}\label{ABsect}

In this section we prove a version of Atiyah-Bott Lemma (Proposition 13.4 \cite{AB2}) for $\mathbb{Z}_{2} $-cohomology. 
 
\begin{definition}\label{def6.wsw}
Let $ E \rightarrow X $ be a $G$-equivariant real vector bundle where $G$ is a compact Lie group. We say $E$ is \textbf{2-primitive} if there exists an elementary abelian 2-subgroup $D_{0}$ of $G$ that acts trivially on $X$ and fixes no nonzero vectors in $E$.
\end{definition}

\begin{proposition}[\textbf{Atiyah-Bott Lemma for $\mathbb{Z}_{2}$-cohomology}]\index{real Atiyah-Bott Lemma for $\mathbb{Z}_{2}$-cohomology}\label{pro6.gay}
Let $G$ be a compact Lie group and  $ \pi:E \rightarrow X $ be a $G$-equivariant vector bundle of rank $m$  over a   connected  manifold $X$.  If $(G,Id)$ has the free extension property and the $G$-vector bundle is 2-primitive, then the equivariant top Stiefel-Whitney class $w_{m}^{G}(E):= w_m(E_G)$ is not a zero-divisor in $ H_{G}^{*}(X;\mathbb{Z}_{2}) $.    
\end{proposition}
\begin{proof}
By the assumptions, the vector bundle $ E\rightarrow X $ is 2-primitive, so there exists an elementary abelian 2-subgroup $D_{0}$ of $G$ that acts trivially on $X$ and fixes no nonzero vectors in $E$. Choose  a maximal elementary abelian 2-subgroup $D$ containing $D_{0}$. The inclusion map   $  i:D \hookrightarrow G $ induces the following commutative diagram of homotopy quotients:

\begin{equation}\label{dia6.wsw}
 \xymatrix{X_{D} \ar[r]^{i_{D}}\ar[d]^{q_{D}} &X_{G}\ar[d]^{q_{G}}  \\
BD  \ar[r]^{Bi} &BG \\} 
 \end{equation}

By functoriality of Stiefel-Whitney classes,
\begin{equation*}
w_{m}^{D}(E)=i_{D}^{*}(w_{m}^{G}(E)).
\end{equation*}
Since $G$ has the free extension property, Proposition \ref{pro5.orb} implies that $ i^{*}_{D}$  is injective. Thus, it suffices to show $ w_{m}^{D}(E)$ is not a zero divisor in $H_D^*(X;\Z_2)$.

Choose a complementary elementary abelian subgroup $ D_{1}\subset D $  such that   $ D=D_{0}\times D_{1} $  where $ D_{1}=(\mathbb{Z}_{2})^{q} $ and  $ p+q=n $. The action of $ D_{0} $ on $X$ is trivial, so we have a homeomorphism
\begin{equation*}
X_{D}\cong BD_{0} \times X_{D_{1}}.
\end{equation*}
By the Kunneth formula,
\begin{equation*}
H^{*}_{D}(X;\mathbb{Z}_{2}) \cong H^{*}(BD_{0};\mathbb{Z}_{2}) \otimes_{\mathbb{Z}_{2}} H^{*}_{D_{1}}(X;\mathbb{Z}_{2}).
\end{equation*}
 This formula  makes   $ H^{*}_{D}(X;\mathbb{Z}_{2}) $ into a bigraded ring.  
Decompose
\begin{equation*}
w_{m}^{D}(E)=\alpha_{0}\otimes 1+\alpha'.
\end{equation*}
where  $ \alpha_{0}\in H^{m}(BD_{0};\mathbb{Z}_{2})$  and $ \alpha'\in \bigoplus_{i=1}^{m} H^{m-i}(BD_{0};\mathbb{Z}_{2}) \otimes_{\mathbb{Z}_{2}} H^{i}_{D_{1}}(X;\mathbb{Z}_{2}) $. Because $H^{*}(BD_{0};\mathbb{Z}_{2}) \cong \Z_2[x_1,\dots,x_p]$ is an integral domain, in order to prove  $w_{m}^{D}(E)$ is not a zero divisor, it suffices to show that $\alpha_0$ is non-zero.

Choose a point $ x\in X $. The inclusion maps $ i_{x}:  \lbrace x\rbrace \hookrightarrow X $  and $ i_{0}:D_{0}\hookrightarrow D $ induce the following pullback diagram:

\begin{equation}\label{dia6.gay}
 \xymatrix{(E_{x})_{D_{0}} \ar[r]\ar[d] &E_{D_{0}}\ar[d]\ar[r] & E_{D}\ar[d]\\
BD_{0} \ar[r]^{Bi_{x}} &X_{D_{0}}\ar[r]^{i_{D_{0}}} & X_{D}\\} 
 \end{equation}

where $ E_{x} $ is the fiber over $ x $. Since $ \alpha_{0} $ is the component  of $ w_{m}^{D}(E) $ in $ H^{*}(BD_{0};\mathbb{Z}_{2}) $, it follows from Diagram \ref{dia6.gay} that
\begin{equation*}
 \alpha_{0}=(i_{D_{0}}\circ Bi_{x})^{*}\Big (w_{m}^{D}(E)\Big ).
 \end{equation*} 
Thus,  $ \alpha_{0} $ is the   top Stiefel-Whitney  class of the vector bundle $ (E_{x})_{D_{0}}\rightarrow BD_{0} $.  Since $ D_{0} $ fixes $ X $, we see that $ E_{x} $ is a representation of $D_{0} \cong \Z_2^p$, which therefore decomposes into a direct sum of 1-dimensional $D_0$-representations $E_x^i$, determining a decomposition into line bundles
\begin{equation*}
 (E_{x})_{D_{0}}=(E_{x}^{1})_{D_{0}} \oplus \cdots \oplus (E_{x}^{m})_{D_{0}}.
\end{equation*}
The Whitney sum formula yields
\begin{equation*}\label{equ6.ali}
\alpha_{0}=w_{m}^{D_{0}}(E_{x})= \prod_{i=1}^{m}  w_{1}^{D_{0}}(E^{i}_{x}),
\end{equation*}
Because $E$ is $2$-primitive, each representation $E_x^i$ is non-trivial, hence the line bundles $(E_{x}^{i})_{D_{0}}$ are nontrivial and 
\begin{equation*}
w_{1}^{D_{0}}(E^{i}_{x}) \neq 0,\ \forall i=1,\dots,m.
\end{equation*}
It follows from  (\ref{equ6.ali}) that $ \alpha_{0}  \neq 0 $. This completes the proof.
\end{proof}

\section{Real Equivariant Perfection}\label{perfectsect}

Let $\mathcal{RH}=(M,\omega,G,\mu,\sigma,\phi)$ be a real Hamiltonian system where $G$ is compact and connected and $f$ is proper.  By Proposition \ref{pro4.zip} there exists a collection of real Hamiltonian  subsystems $ \mathcal{RH}_{I}=(Z_{I},\omega,G_{I},\mu_{I},\sigma,\phi) $  indexed by $\mathcal{I}:= \{ I=(\beta_i, m) \}$ such that the critical set of $f^{\sigma}$ is
\begin{equation}\label{equ7.wsw}
C_{f^{\sigma}}=\coprod_{I} G^{\phi}\times _{G^{\phi}_I} M_I^{\sigma}.
\end{equation}
where $M_I^{\sigma} := \mu^{-1}_I(0) \cap Z_I^{\sigma}$.

\begin{definition}\label{def7.wsw}
 A real Hamiltonian system $ \mathcal{RH} $ is called \textbf{2-primitive} if for every generated subsystem  $ \mathcal{RH}_{I}$, the negative normal bundle of  $M_{I}^{\sigma}$ is a  2-primitive $G^{\phi}_{I}$-bundle (see Definition \ref{def6.wsw}).  
\end{definition}

\begin{theorem}[\textbf{Real Equivariant Perfection}]\index{Real Equivariant Perfection Theorem}\label{the7.wsw}
Let $ (M,\omega,G,\mu,\sigma,\phi) $ be a 2-primitive real Hamiltonian system where $G$ is compact and connected and $f$ is proper. If $(G,\phi)$ has the free extension property, then $f^{\sigma}$ is equivariantly perfect over $ \mathbb{Z}_{2}$. That is
\begin{align}\label{equ7.gal}
\mathbf{P}_t^{G^{\phi}}\boldsymbol{(}M^{\sigma};\mathbb{Z}_{2}\boldsymbol{)}= 
\sum_{I} t^{d_I}\mathbf{P}_t^{G^{\phi}}\boldsymbol{(}C_{I}^\sigma;\Z_2\boldsymbol{)},
\end{align}
where $C_{I}^\sigma   $ are critical subsets of $ f^{\sigma} $ and $ d_I $ is the Morse index of $ f^{\sigma} $ along $ C^{\sigma}_I $. 
 \end{theorem}

\begin{proof}
By  Proposition \ref{pro4.wsw}, the critical set of $f^{\sigma}$ is decomposes into a collection  of disjoint closed $ G^{\phi}$-invariant subsets $C^{\sigma}_{I}$ and there is a smooth invariant Morse stratification such that each stratum $S^{\sigma}_{I}$ has a constant codimension $ d_I $ and deformation retracts onto the corresponding critical subset  $ C^{\sigma}_{I}$. Consider  the generated  real Hamiltonian subsystems $ (Z_{I},\omega,G_{I},\mu_{I},\sigma,\phi) $. By Proposition \ref{pro4.zip}, we have   
\begin{equation}\label{equ7.hkj}
 C^{\sigma}_{I}\cong G^{\phi}\times _{G_{I}} M_I^\sigma.
 \end{equation} 
 
Extend the partial ordering of Proposition \ref{pro4.wsw} (iii) to a total ordering on the index set $\mathcal{I}$ satisfying the condition that $\overline{S^\sigma_{I}} \subseteq \cup_{J \geq I} S^\sigma_{I}$. Define a topological filtration of $M^{\sigma}$ by $S^{\sigma}_{\leq I} := \bigcup _{J \leq I} S_{J}^\sigma$ and let $S^{\sigma}_{< I} := \bigcup _{J < I} S_{J}^{\sigma}$. Let $\nu_I$ be  the normal bundle of $S^{\sigma}_I$ in $S^{\sigma}_{\leq I}$ which we identify with a tubular neighborhood, and let $\nu_I^* = \nu_I \setminus S^{\sigma}_{I} $. Consider the commutative diagram
\begin{equation}\label{equ7.raw}
\xymatrix{ \cdots \ar[r] &  H_{G^{\phi}}^{*} (S^{\sigma}_{\leq I}, S^{\sigma}_{< I};\Z_2) \ar[r]^{\Phi_I} \ar[d]^{\cong}_{\psi_1}  & H_{G^{\phi}}^*(S^{\sigma}_{\leq I};\Z_2) \ar[r] \ar[d] & H_{G^{\phi}}^{*}(S^{\sigma}_{< I};\Z_{2}) \ar[r] \ar[d] & \cdots \\
 \cdots \ar[r] & H_{G^{\phi}}^{*} (\nu_{I}, \nu_I^* ;\Z_2)  \ar[d]^{\cong}_{\psi_2}  \ar[r] & H_{G^{\phi}}^{*} (\nu_{I};\Z_2) \ar[d]^\cong_{\psi_3} \ar[r]  &  H_{G^{\phi}}^{*} (\nu_{I}^*;\Z_2)  \ar[r] & \cdots \\
 & H_{G^{\phi}}^{*-d_I} (S^{\sigma}_{I};\Z_2)  \ar[r]^{\psi_4} & H_{G^{\phi}}^{*} (S^{\sigma}_{I};\Z_2). &  & } 
\end{equation}
In this diagram, the first and second rows are long exact sequences for the pairs $(S^{\sigma}_{\leq I}, S^{\sigma}_{< I})$ and $(\nu_{I}, \nu_I^*)$ respectively, $\psi_1$ is an excision isomorphism, $\psi_2$ is the Thom isomorphism and $\psi_3$ is homotopy equivalence isomorphism. The morphism $\psi_4$ is defined by taking cup product with the top Stiefel-Whitney class $w_m^{G^{\phi}}(\nu_I)$ which plays the role of the ``mod 2 Euler class". Commutativity implies that the map $ \Phi_{I} $ is injective if $ w_m^{G^{\phi}}(\nu_I) $ is not a zero divisor. 

Since  $ S^{\sigma}_{I}$ deformation retracts onto $ C^{\sigma}_{I} \cong G^{\phi} \times_{G^\phi_I}(\mu^{-1}_I(0)^{\sigma})$, it follows that  
\begin{equation}\label{equ7.spt}
 H_{G^{\phi}}^{*}(S^{\sigma}_{I};\mathbb{Z}_{2}) \cong  H_{G^{\phi}_{I}}^{*}( M_I^{\sigma};\mathbb{Z}_{2}).
\end{equation}
The negative normal bundle $\xi_I$ of  $\mu^{-1}_{I}(0)^{\sigma}$ is just the restriction of $\nu_I$, so (\ref{equ7.spt}) sends $w^{G^{\phi}}_m (\nu_{I})$ to $ w^{G^{\phi}_I}_m(\xi_I)$.

By the assumption, $\xi_I$ is 2-primitive with respect to the $G^{\phi}_{I}$-action so by the real Atiyah-Bott Lemma, $ w^{G^{\phi}_I}_m(\xi_I)$ is not a zero divisor and neither is $w^{G^{\phi}}_m  (\nu_{I})$. Applying to (\ref{equ7.raw}) we get short exact sequences
\begin{equation*}\label{equ7.ali}
\xymatrix{ 0\ar[r] &  H_{G^{\phi}}^{*} (S^{\sigma}_{\leq I}, S^{\sigma}_{< I};\Z_2) \ar[r]  & H_{G^{\phi}}^*(S^{\sigma}_{\leq I};\Z_2) \ar[r] & H_{G^{\phi}}^{*}(S^{\sigma}_{< I};\Z_{2}) \ar[r] & 0}
\end{equation*}
yielding identities 
$$  \mathbf{P}_t^{G^\phi} \boldsymbol{(}S^{\sigma}_{\leq I}\boldsymbol{)} = \mathbf{P}_t^{G^\phi} \boldsymbol{(}S^{\sigma}_{< I}\boldsymbol{)} + t^{d_I} \mathbf{P}_t^{G^\phi} \boldsymbol{(}S^{\sigma}_{I}\boldsymbol{)}  $$ 
for all $I$. Applying induction yields 

\begin{align*}
\mathbf{P}_t^{G^{\phi}}\boldsymbol{(}M^{\sigma}\boldsymbol{)}= 
\sum_{I} t^{d_I}\mathbf{P}_t^{G^{\phi}}\boldsymbol{(}S_{I}^\sigma\boldsymbol{)}.
\end{align*}
Combine with $\mathbf{P}_t^{G^{\phi}}\boldsymbol{(}S_{I}^\sigma\boldsymbol{)} = \mathbf{P}_t^{G^{\phi}}\boldsymbol{(}C_{I}^\sigma\boldsymbol{)} = \mathbf{P}_t^{G^{\phi}_I}\boldsymbol{(}M_{I}^\sigma\boldsymbol{)}$ to complete the proof.
\end{proof}

Under the hypotheses of Theorem \ref{the7.wsw} we get a real version of Kirwan Surjectivity:

\begin{corollary}
 The natural map $ \kappa_{\mathbb{R}}: H^{*}_{G^{\phi}}(M^{\sigma};\mathbb{Z}_{2}) \rightarrow H^{*}_{G^{\phi}}(M_{0}^{\sigma};\mathbb{Z}_{2})$ is surjective. 
\end{corollary}

Rearranging (\ref{equ7.gal}) yields 
 
\begin{equation}\label{equ7.ilo}
\mathbf{P}_t^{G^{\phi}}\boldsymbol{(} M_0^\sigma\boldsymbol{)}=\mathbf{P}_t^{G^{\phi}}\boldsymbol{(}M^{\sigma}\boldsymbol{)} -\sum_{I \neq  0} t^{d_I} \mathbf{P}_t^{G^{\phi}_{I}}\boldsymbol{(} M_{I}^{\sigma}\boldsymbol{)}.
\end{equation} 

If $G$ acts freely on $M_0$ then $\mathbf{P}_t^{G^{\phi}}\boldsymbol{(} M_0^\sigma\boldsymbol{)} = \mathbf{P}_t\boldsymbol{(} M^{\sigma}/\!\!/G^{\phi}\boldsymbol{)}$.

\section{Real Equivariant Formality}\label{formalitysect}

A $G$-space $X$ is called \emph{equivariantly formal} with respect to a field $\mathbb{F}$ if the Serre spectral sequence of the fibration $ X \hookrightarrow X_{G} \rightarrow BG $ collapses at the $E_{2}$-term; i.e.
\begin{equation}\label{equ7.tic}
H^{*}_{G}(X;\mathbb{F})= H^{*}(BG;\mathbb{F})\otimes_{\mathbb{F}} H^{*}(X;\mathbb{F}).
\end{equation}
This is equivalent to the Poincar\'{e} series satisfying
\begin{equation}\label{equ7.wiz}
\mathbf{P}_t^G\boldsymbol{(}X;\mathbb{F}\boldsymbol{)}=\mathbf{P}_t\boldsymbol{(}BG;\mathbb{F}\boldsymbol{)}\mathbf{P}_t\boldsymbol{(}X;\mathbb{F}\boldsymbol{)}.
\end{equation}

Kirwan proved that a Hamiltonian system $(M,\omega, G, \mu)$ is equivariantly formal with respect to $ \mathbb{Q} $ if \underline{both} $G$ and $M$ are  compact and connected. For a real abelian Hamiltonian system, Biss-Guillemin-Holm (see \cite{BGH}, Theorem B) proved the following result.

\begin{proposition}[\textbf{Biss-Guillemin-Holm}]\label{pro7.wsw}
Let $ (M,\omega, T, \mu, \sigma, \phi) $ be a real Hamiltonian system in which $M$ is a compact connected manifold, $T$ is a torus and $ \phi(g)=g^{-1} $ for any $ g\in T $. If $M^{\sigma}$ is the real locus and $ T^{\phi} $ is the real subgroup, then
 \begin{equation}
 H^{*}_{T^{\phi}}(M^{\sigma};\mathbb{Z}_{2}) \cong H^{*}(BT^{\phi};\mathbb{Z}_{2}) \otimes_{\mathbb{Z}_{2}} H^{*}(M^{\sigma};\mathbb{Z}_{2})
 \end{equation}
 as graded modules over $ H^{*}(BT^{\phi};\mathbb{Z}_{2}) $.
 \end{proposition}

We use this result to prove a real version of equivariant formality for nonabelian real Hamiltonian systems. First we have the following definition.

\begin{definition}\label{def7.gay}
We say that a compact connected involutive Lie group $ (G,\phi) $ has the \textbf{toral free extension property}\index{toral free extension property} if the following hold.
\begin{enumerate}
\item  The pair $ (G,\phi) $ has the free extension property.
\item There exists a maximal torus $T\leq G$ such that $ \phi(g)=g^{-1} $ for all $ g\in T $ and $ T^{\phi} $ is a maximal elementary abelian 2-subgroup of $ G^{\phi} $.
\end{enumerate}
\end{definition}

Both $(\mathrm{U}(n), \phi)$ and $(\mathrm{SU}(n),\phi)$ where $\phi(A) = \overline{A}$ satisfy the toral free extension party. However $(\mathrm{U}(n) \times \mathrm{U}(n), \phi)$ with $\phi(A,B) = (\overline{B}, \overline{A})$ does not.

\begin{theorem}[\textbf{Real Equivariant Formality}]\index{Real Equivariant Formality Theorem}\label{the7.gay}
Let $ (M,\omega, G,\mu,\sigma,\phi) $ be a   real Hamiltonian system where both $G $ and $M$ are  compact and connected. If $(G,\phi) $ has the toral free extension property, then the real locus $M^{\sigma}$ is $G^{\phi}$-equivariantly formal with respect to the field $ \mathbb{Z}_{2} $. 
\end{theorem}

\begin{proof}
Since $(G,\phi) $ has the toral free extension property, there exists a maximal torus $T\leq G$ such that $ \phi(g)=g^{-1} $ for any $ g\in T $.  Let  $ \mu_{T}: M \rightarrow \mathfrak{t}^{*}$ be the composition of $ \mu $ with the orthogonal projection $\lie{g} \rightarrow \lie{t}^{*} $. Then $ (M,\omega, T,\mu_{T},\sigma,\phi_{T}) $ is a real Hamiltonian system such that $ \phi_{T}: T \rightarrow T $ is the inversion map. By the Biss-Guillemin-Holm theorem, $M^{\sigma}$ is $T^{\phi}$-equivariantly formal yielding 
\begin{equation}\label{equ7.vol}
\mathbf{P}_t\boldsymbol{(} M^{\sigma}_{T^{\phi}}\boldsymbol{)} = \mathbf{P}_t\boldsymbol{(}M^\sigma\boldsymbol{)} \mathbf{P}_t\boldsymbol{(}BT^\phi\boldsymbol{)}.
\end{equation}

Consider the following  commutative diagram:
\begin{equation*} 
\xymatrix{     &     & M^{\sigma} \ar[rd]^{i_{2}} \ar[ld]_{i_{1}}  & \\
   & M^{\sigma}_{T^{\phi}}\ar[dd]\ar[rr]   &      &  M^{\sigma}_{G^{\phi}} \ar[dd]  &  \\
    G^{\phi}/T^{\phi}\ar[ru]^{j_{1}}\ar[rd]_{j_{2}}   &  &      &   &  \\ 
    & BT^{\phi} \ar[rr] &      &  BG^{\phi} &  } 
\end{equation*}

Since  $ (G,\phi) $ has the free extension property and $T^{\phi}$ is a maximal elementary abelian 2-subgroup of $G^{\phi}$, the induced maps $j_1^*$ and $j_2^*$ are surjective. By the Leray-Hirsch Theorem we deduce
\begin{eqnarray*}\label{equ7.diy}
\mathbf{P}_t(BT^{\phi})  &=& \mathbf{P}_t(BG^{\phi}) \mathbf{P}_t\boldsymbol{(}G^{\phi}/T^{\phi}\boldsymbol{)} \text{, and}\\
\mathbf{P}_t\boldsymbol{(}M^{\sigma}_{T^{\phi}}\boldsymbol{)} & =& \mathbf{P}_t\boldsymbol{(}M^{\sigma}_{G^\phi}\boldsymbol{)} \mathbf{P}_t\boldsymbol{(}G^{\phi}/T^{\phi}\boldsymbol{)}.
\end{eqnarray*}
Combining these with (\ref{equ7.vol}), we get
\begin{equation*}
\mathbf{P}_t\boldsymbol{(}M^{\sigma}_{G^{\phi}}\boldsymbol{)}\mathbf{P}_t\boldsymbol{(}G^{\phi}/T^{\phi}\boldsymbol{)}=\mathbf{P}_t\boldsymbol{(}M^{\sigma}\boldsymbol{)} \mathbf{P}_t\boldsymbol{(}BG^{\phi}\boldsymbol{)}\mathbf{P}_t\boldsymbol{(}G^{\phi}/T^{\phi}\boldsymbol{)},
\end{equation*}
which  implies that $\mathbf{P}_t\boldsymbol{(}M^{\sigma}_{G^{\phi}}\boldsymbol{)} =\mathbf{P}_t\boldsymbol{(}M^{\sigma}\boldsymbol{)}\mathbf{P}_t\boldsymbol{(}BG^{\phi}\boldsymbol{)} $ completing the proof.
\end{proof}

\section{Examples}\label{examplessect}

\begin{example}\label{ConstructTheGrass}

Let $M = \mathrm{Mat}_{n \times k}(\C) \cong \C^{nk}$ be the set of $n \times k$-matrices with the standard symplectic form (see Example \ref{exa3.orb}). The action of $\mathrm{U}(k)$ by right multiplication is Hamiltonian with moment map that sends a $n \times k$ matrix $A$ to $\mu(A) = \frac{\sqrt{-1}}{2} (\overline{A}^t A - Id_k)$. Observe that $\mu^{-1}(0)$ equals the set $n \times k$-matrices whose columns form an orthonormal frame in $\C^n$, therefore $\mathrm{U}(k)$ acts freely on $\mu^{-1}(0)$ and 
$$ \mu^{-1}(0)/\mathrm{U}(k) = \mathrm{Gr}_k(\C^n). $$
This $\mathrm{U}(k)$-action extends to a $\mathrm{GL}_k(\C)$-action by right multiplication and $\mu^{-1}(0)/\mathrm{U}(k)$ corresponds to a GIT quotient by the Kempf-Ness Theorem. This implies that minimum Morse stratum $S_0$ is equal to the union of $\mathrm{GL}_k(\C)$-orbits intersecting $\mu^{-1}(0)$ for which it follows that if $A$ has full rank $k$ then it is a critical point of $f$ if and only if it lies in $\mu^{-1}(0)$.

Suppose that $A$ is critical point of $f$ and has rank $r < k$. Then after acting by some element of $\mathrm{U}(k)$ if necessary, we can assume that the first $r$ columns of $A$ are linearly independent and the remaining columns are zero. Effectively then, we can regard $A$ as an element of $\mathrm{Hom}( \C^r, \C^n)$ and deduce that $A$ is a critical point for $f$ if and only if the first $r$ columns of $A$ are orthonormal. The corresponding critical value is the adjoint orbit containing the diagonal matrix $ \beta$ with $r$ entries equal to $0$ and the rest equal to $\sqrt{-1}/2$. The corresponding Morse stratum is the set of all matrices of rank $r$. Therefore the Morse stratification agrees with the decomposition $M = S_0 \cup S_1 \cup \cdots \cup S_k$ where $S_i$ is the set of matrices of rank $k-i$. It is straightforward to calculate that the dimension of $S_i$ is $ 2(k-i)i + 2(k-i)(n)= 2(k-i)(n+i)$ and therefore has codimension $2kn - 2(k-i)(n+i) = 2i(n-k+i)$. The Hamiltonian subsystem corresponding to $S_i$ is isomorphic to $(\mathrm{Mat}_{n \times (k-i)}, \omega, \mathrm{U}(k-i) \times \mathrm{U}(i), \mu_i)$ which looks like a lower rank version of $M$ times an extra trivial action by $\mathrm{U}(i)$. Applying perfection yields the formula
$$   \mathbf{P}_t^{\mathrm{U}(k)}\boldsymbol{(}\mu^{-1}(0);\Q\boldsymbol{)} =   \mathbf{P}_t\boldsymbol{(}B\mathrm{U}(n);\Q\boldsymbol{)} - \sum_{i=1}^k t^{2i(n-k+i)}  \mathbf{P}_t^{\mathrm{U}(k-i)}\boldsymbol{(}\mu_i^{-1}(0);\Q\boldsymbol{)} \mathbf{P}_t\boldsymbol{(}B\mathrm{U}(i);\Q\boldsymbol{)}. $$
Inputting known values determines recursion relation for the Poincar\'e polynomials of the complex Grassmannians.
$$  \mathbf{P}_t\boldsymbol{(}\mathrm{Gr}_k(\C^n);\Q\boldsymbol{)} = \prod_{p=1}^k \frac{1}{1-t^{2p}} - \sum_{i=1}^k   \mathbf{P}_t\boldsymbol{(}\mathrm{Gr}_{k-i}(\C^n);\Q\boldsymbol{)}   \prod_{p=1}^i \frac{t^{2(n-k+i)} }{1-t^{2p}} . $$
This system admits a real structure where $\sigma$ and $\phi$ are the standard conjugations. We get $M^\sigma = \mathrm{Mat}_{n \times k} (\R)$ and $\mathrm{U}(k)^{\phi} = \mathrm{O}(k)$. Running through the details produces an analogous recursive formula
$$   \mathbf{P}_t\boldsymbol{(}\mathrm{Gr}_k(\R^n);\Z_2\boldsymbol{)} = \prod_{p=1}^k \frac{1}{1-t^{p}} - \sum_{i=1}^k  \mathbf{P}_t\boldsymbol{(}\mathrm{Gr}_{k-i}(\R^n);\Z_2\boldsymbol{)} \prod_{p=1}^i \frac{t^{(n-k+i)} }{1-t^{p}}.  $$
The solution of this recursion is the well known formula $ \mathbf{P}_t\boldsymbol{(}\mathrm{Gr}_k(\R^n)\boldsymbol{)} = \frac{\prod_{p=n-k+1}^n (1-t^p)}{\prod_{p=1}^k(1-t^p)}$. Observe that even without solving the recursion, we can deduce immediately that 
\begin{equation}\label{chekst}
 \mathbf{P}_t\boldsymbol{(}\mathrm{Gr}_k(\R^n);\Z_2\boldsymbol{)}  = \mathbf{P}_{t^{1/2}}\boldsymbol{(}\mathrm{Gr}_k(\C^n);\Q\boldsymbol{)}.
 \end{equation}

\end{example}

\begin{example}\label{prodgrassex}

Kirwan considers the diagonal action of $\mathrm{U}(n)$ on a product of Grassmannians $M=\mathrm{Gr}_{l_{1}}(\mathbb{C}^{n})\times \cdots\times \mathrm{Gr}_{l_{r}}(\mathbb{C}^{n})$. This action is Hamiltonian with moment map 
\begin{equation}\label{equ3.dow}
\mu(V_{1},\dots,V_{r})=\sqrt{-1}\Big(\sum_{j=1}^{r} \mathrm{Pr}_{V_{j}} - \Big(\dfrac{\sum_{j=1}^{r}l_{j}}{n}\Big)\mathrm{Id}_{n}\Big),
\end{equation}
where $\mathrm{Pr}_{V_{j}}$ is the orthogonal projection on $V_j$. The scalar matrices $\mathrm{U}(1) \leq \mathrm{U}(n)$ act trivially on $M$ determining a Hamiltonian action by $\mathrm{PU}(n) = \mathrm{U}(n)/\mathrm{U}(1)$ on $M$.  If $gcd(n,l_1+\dots+l_r) = 1$, then $\mathrm{PU}(n)$ acts freely on $M_0$ so $M/\!\!/\mathrm{U}(n) = M/\!\!/\mathrm{PU}(n) $ is a manifold and 
\begin{equation}\label{UtoPU}
 \mathbf{P}_t^{\mathrm{U}(n)}\boldsymbol{(}M_0;\mathbb{Q}\boldsymbol{)} = \mathbf{P}_t\boldsymbol{(}B\mathrm{U}(1);\mathbb{Q}\boldsymbol{)} \mathbf{P}_t^{\mathrm{PU}(n)}\boldsymbol{(}M_0;\mathbb{Q}\boldsymbol{)} =  \frac{1}{1-t^2} \mathbf{P}_t\boldsymbol{(}M/\!\!/\mathrm{U}(n);\mathbb{Q}\boldsymbol{)}.
 \end{equation}

We have a real pair $(\phi,\sigma)$ where $\phi$ and $\sigma$ are standard complex conjugations. Therefore $\mathrm{U}(n)^\phi = \mathrm{O}(n)$ and 
\begin{equation}\label{forrefgrr}
M^\sigma  =\mathrm{Gr}_{l_{1}}(\mathbb{R}^{n})\times \cdots\times \mathrm{Gr}_{l_{r}}(\mathbb{R}^{n}).
\end{equation}  
 If $gcd(n,l_1+\dots+l_r) = 1$, then $M/\!\!/\mathrm{O}(n) = M/\!\!/\mathrm{PO}(n) $ is a Lagrangian submanifold of $M/\!\!/\mathrm{U}(n) = M/\!\!/\mathrm{PU}(n)$.

\begin{proposition}
We have equality
$$  \mathbf{P}_t^{\mathrm{O}(n)}\boldsymbol{(}M_0^{\sigma} ;\Z_2\boldsymbol{)} = \mathbf{P}_{t^{1/2}}^{\mathrm{U}(n)}\boldsymbol{(}M_0;\mathbb{Q}\boldsymbol{)}. $$
\end{proposition}

\begin{proof}

Following the general scheme, there exist a collection of  Hamiltonian subsystems   $(Z_I, \omega_I, \mathrm{U}(n)_I, \mu_I)$  such that 
\begin{equation}
\mathbf{P}_t^{\mathrm{U}(n)}\boldsymbol{(}M_0);\mathbb{Q}\boldsymbol{)}=\mathbf{P}_t(M;\mathbb{Q})\mathbf{P}_t\boldsymbol{(}B\mathrm{U}(n);\mathbb{Q}\boldsymbol{)}-\sum_{I \neq 0} t^{2d_I}\mathbf{P}_t^{\mathrm{U}(n)_{I}}\boldsymbol{(}M_{I};\mathbb{Q}\boldsymbol{)},
\end{equation}
where $M_I := \mu^{-1}(0)$. It is more convenient to refine the index set described in \S \ref{MSfRH} to one whose strata are connected.  

Identify the Lie algebra of the maximal torus $\lie{t} \cong \R^n$ in the standard way, with positive Weyl chamber $\lie{t}_+ = \{(x_1,\dots,x_n) | x_1 \geq \dots \geq x_n \geq 0\}$.  The index set $\mathcal{I}$ is a finite subset of $\lie{t}_+ \times \mathrm{Mat}_{s,r}(\mathbb{Z})$ of pairs $(\beta, l)$ for which $$\beta = \Big( \underbrace{\frac{k_{1}}{m_{1}},\dots, \frac{k_{1}}{m_{1}}}_{m_{1}},\dots,\underbrace{\frac{k_{s}}{m_{s}},\dots, \frac{k_{s}}{m_{s}}}_{m_{s}}\Big )$$ where $m_i, k_i \in \Z$, $m_i >0$, $k_i \geq 0$, $m_1+\dots+m_s = n$, and $k_1+\dots+k_s = \sum_{j=1}^{r}l_{j}$, $ \frac{k_{1}}{m_{1}}> \cdots >\frac{k_{s}}{m_{s}}$  and $l$ is a matrix of non-negative integers $ l=(l_{i,j})\in \mathrm{Mat}_{s,r}(\mathbb{Z})$ such that  $\sum_{i=1}^{s}l_{ij}=l_{j}$ and  $\sum_{j=1}^{r}l_{ij}=k_{i}$. Then
$$Z_{\beta,l}  \cong \prod_{i=1}^s \prod_{j=1}^r \mathrm{Gr}_{l_{i,j}}(\C^{m_i})$$ 
and the Hamiltonian subsystem $(Z_{\beta,l}, \omega_{\beta,l}, \mathrm{U}(n)_{\beta}, \mu_{\beta,l})$ is a product of Hamiltonian systems of the type (\ref{equ3.dow}) but with rank $m_i$ less than $n$. Consequently
\begin{equation}\label{recrU}
\mathbf{P}_t^{\mathrm{U}(n)}\boldsymbol{(}M_0;\mathbb{Q}\boldsymbol{)}=\mathbf{P}_t\boldsymbol{(}M;\mathbb{Q}\boldsymbol{)}\mathbf{P}_t\boldsymbol{(}B\mathrm{U}(n);\mathbb{Q}\boldsymbol{)}-\sum_{(\beta,l) \neq 0} t^{2d_{\beta,l}} \prod_{i=1}^s \mathbf{P}^{\mathrm{U}(m_i)}_t\boldsymbol{(} M_{\beta,l};\Q\boldsymbol{)}
\end{equation}
which can be calculated recursively in the rank $n$. The base of the recursion occurs with $n=1$ in which case $\mathrm{Gr}_0(\C^1) = \mathrm{Gr}_1(\C^1)$ is a point and $\mathbf{P}_t^{\mathrm{U}(1)}(point) = \mathbf{P}_t(B\mathrm{U}(1)) = (1-t^2)^{-1}$. 

The calculation for the real quotient proceeds analogously.  The real Hamiltonian subsystems satisfy
$$Z_{\beta,l}^{\sigma}  \cong \prod_{i=1}^s \prod_{j=1}^r \mathrm{Gr}_{l_{i,j}}(\R^{m_i})$$ 
and are products of systems of the form (\ref{forrefgrr}) but of rank $m_i < n$. The action of $\mathrm{U}(n)_{\beta}$ on the tangent spaces $T_x M_{\beta,l}$ are $\Z$-primitive as observed by Kirwan (Remark 16.11 \cite{KIR}), so the system is 2-primitive. Applying Theorem \ref{the7.wsw} we obtain the recursive formula
\begin{equation}\label{recrO}
\mathbf{P}_t^{\mathrm{O}(n)}\boldsymbol{(}M_0^{\sigma} ;\Z_2\boldsymbol{)}=\mathbf{P}_t\boldsymbol{(}M^{\sigma} ;\Z_2\boldsymbol{)}\mathbf{P}_t\boldsymbol{(}B\mathrm{O}(n);\Z_2\boldsymbol{)}-\sum_{(\beta,l) \neq 0} t^{d_{\beta,l}} \prod_{i=1}^s \mathbf{P}_t^{\mathrm{O}(m_i)}\boldsymbol{(} M_{\beta,l}^{\sigma};\Z_2 \boldsymbol{)}.
\end{equation}
The inputs to the recursions (\ref{recrU}) and (\ref{recrO}) satisfy 
\begin{eqnarray*}
\mathbf{P}_t\textbf{(}M^{\sigma} ;\Z_2\textbf{)}  &=& \mathbf{P}_{t^{1/2}}\textbf{(} M; \mathbb{Q}\textbf{)}, \\
\mathbf{P}_t\textbf{(}B\mathrm{O}(n);\Z_2\textbf{)} &=& \mathbf{P}_{t^{1/2}}\textbf{(}B\mathrm{U}(n);\Q\textbf{)},\\
   \mathbf{P}_t\textbf{(}B\mathrm{O}(1);\Z_2\textbf{)} &= &  \mathbf{P}_{t^{1/2}}\textbf{(}B\mathrm{U}(1); \mathbb{Q}\textbf{)}.
\end{eqnarray*}
We conclude that 
$$ \mathbf{P}_t^{\mathrm{O}(n)}\boldsymbol{(}M_0^{\sigma} ;\Z_2\boldsymbol{)} = \mathbf{P}^{\mathrm{U}(n)}_{t^{1/2}}\boldsymbol{(}M_0;\mathbb{Q}\boldsymbol{)}. $$
\end{proof}

\begin{corollary}
If $gcd(n,l_1+\dots+l_r) = 1$ and $n$ is odd, then 
\begin{equation}\label{halfbetti}
\mathbf{P}_{t}\boldsymbol{(}M^{\sigma}/\!\!/\mathrm{O}(n);\Z_2\boldsymbol{)} = \mathbf{P}_{t^{1/2}}\boldsymbol{(}M/\!\!/\mathrm{U}(n);\mathbb{Q}\boldsymbol{)}.
\end{equation}
\end{corollary}

\begin{proof}
Since $n$ is odd, the short exact sequence $1 \rightarrow \mathrm{O}(1) \rightarrow  \mathrm{O}(n) \rightarrow  \mathrm{PO}(n) \rightarrow 1$ splits, yielding an isomorphism $\mathrm{O}(n) \cong \mathrm{O}(1) \times \mathrm{PO}(n)$, sending $\mathrm{SO}(n)\cong \mathrm{PO}(n)$. Since $\mathrm{O}(1)$ acts trivially and $\mathrm{PO}(n)$ acts freely, we get $$\mathbf{P}^{O(n)}_t\boldsymbol{(}M_0^{\sigma}\boldsymbol{)} = \mathbf{P}_t^{\mathrm{PO}(n)}\boldsymbol{(}M_0^\sigma\boldsymbol{)} \mathbf{P}_t\boldsymbol{(}B\mathrm{O}(1)\boldsymbol{)} = \frac{1}{1-t}\mathbf{P}_t\boldsymbol{(}M_0^\sigma / \mathrm{PO}(n) \boldsymbol{)} = \frac{1}{1-t}\mathbf{P}_t\boldsymbol{(}M^\sigma /\!\!/\mathrm{O}(n) \boldsymbol{)}.$$
Compare with (\ref{UtoPU}).
\end{proof}

\end{example}


The phenomenon revealed in (\ref{chekst}) and (\ref{halfbetti}) is a familiar one in real symplectic geometry (see for example \cite{HHP}). The next example shows that (\ref{halfbetti}) does not always hold for real quotients.

\begin{example}\label{exa8.hkj}
Consider $ (\mathbb{CP}^{1},\omega_{FS},\mathrm{U}(1),\mu_{1},\sigma_{\C P^1},\phi) $ a special case of Example \ref{prodgrassex} and let $(M,\omega_M ,\sigma_M)$ be any real symplectic manifold. Define a new real Hamiltonian system $ (\mathbb{CP}^{1} \times M ,\omega, \mathrm{U}(1),\mu, \sigma ,\phi) $ where $ \omega := \omega_{FS} + \omega_M$ is the product symplectic form, $\sigma:= \sigma_{\C P^1} \times \sigma_M$, and $ \mu(x,y)=\mu_1(x) $. The preimage $\mu^{-1}(0) = S^1 \times M$ where $S^1$ is the equator in $\mathbb{CP}^1$ so the symplectic quotient is $(\mathbb{CP}^{1} \times M)/\!\!/\mathrm{U}(1)\cong M $ and the real quotient is $ (\mathbb{CP}^{1} \times M)^{\sigma}/\!\!/\mathrm{O}(1)\cong M^{\sigma_M}$. Thus whenever $M$ has non-trivial $\Q$-Betti numbers in odd degree we have $\mathbf{P}_t\boldsymbol{(}M^{\sigma_M} ;\Z_2\boldsymbol{)} \neq \mathbf{P}_{t^{1/2}}\boldsymbol{(}M;\Q\boldsymbol{)}$ and (\ref{halfbetti}) is not satisfied.
\end{example}

\end{document}